\documentclass[oneside,english,a4wide]{amsart}
\usepackage[latin9]{inputenc}
\usepackage[a4paper]{geometry}
\usepackage{mathabx}
\usepackage{esvect}
\usepackage{babel}
\usepackage{mathtools}
\usepackage{amstext}
\usepackage{amsthm}
\usepackage{amssymb}
\usepackage{enumerate}
\usepackage{esint}
\usepackage{graphicx}
\usepackage{bbm}
\usepackage[all]{xy}
\usepackage{mathrsfs}
\usepackage{fancyhdr}
\usepackage[unicode=true,
bookmarks=true,bookmarksnumbered=true,bookmarksopen=true,bookmarksopenlevel=2,
breaklinks=false,pdfborder={0 0 1},backref=false,colorlinks=false]
{hyperref}
\hypersetup{
colorlinks,linkcolor=blue,anchorcolor=blue,citecolor=blue}
\usepackage{breakurl}

\usepackage[svgnames]{xcolor}  

\usepackage{graphicx}

\usepackage[normalem]{ulem} 
\usepackage{soul}
\makeatletter

\numberwithin{equation}{section}
\numberwithin{figure}{section}
\theoremstyle{plain}
\newtheorem{thm}{\protect\theoremname}[section]
\theoremstyle{plain}

\newtheorem{proposition}[thm]{Proposition}
\newtheorem{corollary}[thm]{Corollary}
\theoremstyle{definition}

\theoremstyle{plain}
\newtheorem{lem}[thm]{\protect\lemmaname}
\newtheorem{lemma}[thm]{Lemma}
\theoremstyle{remark}
\newtheorem{rem}[thm]{\protect\remarkname}

\theoremstyle{definition}
\newtheorem{definition}[thm]{Definition}
\newtheorem*{claim*}{Claim}

\theoremstyle{remark}
\newtheorem{remark}[thm]{Remark}
\theoremstyle{definition}
\newtheorem*{defn*}{\protect\definitionname}

\usepackage{babel}
\usepackage{amscd}

\makeatother

\providecommand{\definitionname}{Definition}
\providecommand{\lemmaname}{Lemma}
\providecommand{\propositionname}{Proposition}
\providecommand{\remarkname}{Remark}
\providecommand{\theoremname}{Theorem}

\newcommand{\Rmnum}[1]{\expandafter\@slowromancap\romannumeral #1@}
\makeatother

\newcommand{\real}{{\mathbb R}}
\newcommand{\nat}{{\mathbb N}}
\newcommand{\ent}{{\mathbb Z}}
\newcommand{\com}{{\mathbb C}}
\newcommand{\ce}{{\mathbb E}}

\newcommand{\A}{{\mathcal A}}
\newcommand{\B}{{\mathcal B}}
\newcommand{\D}{{\mathscr{D}}}

\newcommand{\F}{{\mathcal F}}
\renewcommand{\H}{ {\mathcal H} }
\renewcommand{\L}{{\mathcal L}}
\newcommand{\M}{{\mathcal M}}

\newcommand{\8}{\infty}

\newcommand{\la}{\langle}
\newcommand{\ra}{\rangle}

\newcommand{\be}{\begin{eqnarray*}}
\newcommand{\ee}{\end{eqnarray*}}
\newcommand{\beq}{\begin{equation}}
\newcommand{\eeq}{\end{equation}}
\newcommand{\beqn}{\begin{equation*}}
\newcommand{\eeqn}{\end{equation*}}
\newcommand{\bs}{\begin{split}}
\newcommand{\es}{\end{split}}

\newcommand\norm[1]{ \left\| #1 \right\| }
\newcommand\jdz[1]{ \left| #1 \right| }
\newcommand\sk[1]{ \left( #1 \right) }
\newcommand\mk[1]{ \left[ #1 \right] }
\newcommand\lk[1]{ \left\{ #1 \right\} }

\newcommand\supp{\text{supp}}

\renewcommand\st[1]{}

\begin{document}
	
\title[Multiplication between elements in Martingale Hardy spaces and their duals]{Multiplication between elements in Martingale Hardy spaces and their dual spaces}


\author[O. Bakas]{Odysseas Bakas} 
\address[O.\ Bakas]{BCAM - Basque Center for Applied Mathematics, 48009 Bilbao, Spain}
\email{obakas@bcamath.org}
\thanks{O. Bakas is partially supported by the projects CEX2021-001142-S, RYC2018-025477-I, PID2021-122156NB-I00/AEI/10.13039/501100011033 funded by Agencia Estatal de Investigaci\'on and acronym ``HAMIP'', Juan de la Cierva Incorporaci\'on IJC2020-043082-I and grant BERC 2022-2025 of the Basque Government.}

\author[Z. Xu]{Zhendong Xu}
\address[Z. Xu]{Laboratoire de Math{\'e}matiques, Universit{\'e} de Bourgogne Franche-Comt{\'e}, 25030 Besan\c{c}on Cedex, France}
\email{xu.zhendong@univ-fcomte.fr}

\author[Y. Zhai]{Yujia Zhai}
\address[Y. Zhai]{School of Mathematical and Statistical Sciences, Clemson University, 29634 South Carolina, USA}
\email{zhai@clemson.edu}
\thanks{Y. Zhai acknowledges partial support from ERC project FAnFArE no. 637510 and the region Pays de la Loire.}

\author[H. Zhang]{Hao Zhang}
\address[H. Zhang]{Laboratoire de Math{\'e}matiques, Universit{\'e} de Bourgogne Franche-Comt{\'e}, 25030 Besan\c{c}on Cedex, France}
\email{hao.zhang@univ-fcomte.fr}

%\date{}

\maketitle

\begin{abstract} In this paper, we establish continuous bilinear decompositions that arise in the study of products between elements in martingale Hardy spaces $ H^p\ (0<p\leqslant 1) $ and functions in their dual spaces. Our decompositions are based on martingale paraproducts. As a consequence of our work, we also obtain analogous results for dyadic martingales on spaces of homogeneous type equipped with a doubling measure.
\end{abstract}

%%%%%%%%%%%%%%%%%%%%%%%%%%%%%%%%%%%%%%%%%%%%%%%%%%%%%%
%%%%%%%%%%%%%%%%%%%%%%%%%%%%%%%%%%%%%%%%%%%%%%%%%%%%%%

\section{Introduction}\label{Introduction}

The pointwise product of a function in the classical Hardy space $H^1(\mathbb{R}^n)$ and a function of bounded mean oscillation on $\mathbb{R}^n$ need not be in $L^1(\mathbb{R}^n)$; see e.g. \S 6.2 in Chapter IV in \cite{Stein}. However, using Fefferman's duality theorem \cite{F} and the fact that the pointwise product of a $BMO$-function and a $C_0^{\infty}$-function is in $BMO(\mathbb{R}^n)$, Bonami, Iwaniec, Jones and Zinsmeister defined in \cite{Bo1} the product $f \times g$ of a function $f \in H^1(\mathbb{R}^n)$ and a function $g \in BMO(\mathbb{R}^n)$ as a distribution given by
\begin{equation}\label{def_prod}
 \langle f \times g, \phi \rangle := \langle g \cdot \phi, f  \rangle, \quad \phi \in C_0^{\infty} (\mathbb{R}^n) ,
\end{equation}
where in the right-hand side of \eqref{def_prod} the duality between  $f \in H^1(\mathbb{R}^n)$ and $g \cdot \phi \in BMO(\mathbb{R}^n)$ is employed. Moreover, it is shown in \cite{Bo1} that for any fixed $f \in H^1(\mathbb{R}^n)$ there exist two linear continuous operators $S_f$ from $BMO(\mathbb{R}^n)$ to $L^1(\mathbb{R}^n)$ and $T_f$ from $BMO (\mathbb{R}^n)$ to a weighted Hardy--Orlicz space such that
$$  f \times g= S_f(g)+T_f(g)  $$
for all  $g\in BMO(\real^n)$; see \cite[Theorem 1.6]{Bo1}.

In \cite{Bo}, using wavelet analysis, Bonami, Grellier and Ky showed that there exist two bilinear continuous operators $S$ from $H^1(\mathbb{R}^n)\times BMO(\mathbb{R}^n)$ to $L^1(\mathbb{R}^n)$ and $T$ from $H^1(\mathbb{R}^n)\times BMO(\mathbb{R}^n)$ to $ H^{\log}(\real^n) $ such that
$$ f \times g= S(f,g)+T(f,g)  $$
 for all  $ f \in H^1(\real^n) $ and for all $g\in BMO(\real^n)$; see \cite[Theorem 1.1]{Bo}.  The Musielak Hardy--Orlicz space $H^{\log}(\real^n) $ is defined as the class consisting of all distributions $ h $ on $\mathbb{R}^n$ whose grand maximal function $ \M h $ satisfies
$$
\int_{\real^n} \frac{|\M h (x)|}{\log(e+|x|)+\log(e+|\M h (x)|)} dx < \infty
$$
and is smaller than the weighted Hardy--Orlicz space appearing in \cite{Bo1}. In fact, as explained in \cite{Bo1}, in view of the results of Nakai and Yabuta \cite{NY} on pointwise multipliers of $BMO(\real^n)$ and duality, the Musielak Hardy--Orlicz space $H^{\log}(\real^n) $ is optimal.

In addition, continuous bilinear decomposition theorems for products of elements in $H^p(\real^n)$, for $0<p<1$,  and their dual spaces were established in \cite{BoY}.  

Using the theory of wavelets on spaces of homogeneous type, which was developed by  Auscher and Hyt\"onen in \cite{AH}, the aforementioned results have been extended to spaces of homogeneous type by Liu, Yang and Yuan \cite{Liu} and  Xing, Yang and  Liang \cite{Xi}. More precisely, in \cite{Liu} and \cite{Xi}, continuous bilinear decompositions for products between elements in atomic Hardy spaces $H^p_{\rm at}(\Omega)$ (in the sense of Coifman and Weiss \cite{Co}) and their dual spaces were established in the case where $p\in (\frac{n}{n+1}, 1]$. Here $n$ is defined as the dimension of the homogeneous space $\Omega$.  

Recently, in \cite{BPRS}, a dyadic variant of the aforementioned results of Bonami, Grellier, and Ky was established; see \cite[Theorem 24]{BPRS}, which in turn was used to deduce a periodic version of \cite[Theorem 1.1]{Bo}; see \cite[Theorem 28]{BPRS}.

Motivated by \cite{BPRS}, the first part of this article is concerned with the study of multiplication between Hardy spaces and their dual spaces for martingales on a probability space $\Omega$. More specifically, we study multiplications between functions in the martingale Hardy space $H^1(\Omega)$ and its dual space $BMO(\Omega)$ as stated in our first result, Theorem \ref{Theorem A}. We also investigate the case $0<p<1$, namely multiplication between elements in $H^p(\Omega)$ and their dual spaces, the so-called martingale Lipschitz spaces $ \Lambda_1(\alpha_p)$ with $ \alpha_p := \frac{1}{p} -1 $, see Theorem \ref{Theorem C}. Since the dual space $(H^p(\Omega))^*$ could be $\{0\}$ for some irregular martingales, we shall only consider regular martingales where every $\sigma-$algebra $\F_k$ in the corresponding filtration is generated by countably many atoms.

We would like to mention that Yong Jiao, Guangheng Xie, Dachun Yang, and Dejian Zhou have independently obtained Theorem \ref{Theorem A}, and derived from it interesting applications on the boundedness of operators involving commutators in \cite{JXYZ}.

\begin{thm}\label{Theorem A} Let $(\Omega, \F, P)$ be a probability space equipped with the filtration $\{\F_k\}_{k \ge 1}$. 
	
There exist continuous bilinear operators $\Pi_1:H^1(\Omega)\times BMO(\Omega) \to L^1(\Omega), \ \Pi_2:H^1(\Omega)\times BMO(\Omega) \to H^1(\Omega)$ and $ \Pi_3:H^1(\Omega)\times BMO(\Omega) \to H^{\Phi}(\Omega) $  such that
\begin{equation*}\label{A}
  f\cdot g = \Pi_1(f,g)+\Pi_2(f,g) + \Pi_3(f,g)
\end{equation*}
for all $ f\in H^1(\Omega)$ and $g\in BMO(\Omega)$,  
where $f\cdot g$ is in the sense of the pointwise multiplication.
\end{thm}

In Theorem \ref{Theorem A}, $H^\Phi(\Omega)$ is a martingale Hardy--Orlicz space defined in terms of the growth function $\Phi (t) $; see Definition \ref{def_M-H-O} and \eqref{Phi}  below. We shall refer to the terms $\Pi_2(f,g)$ and $\Pi_3(f,g)$ as the martingale paraproducts. 

Theorem \ref{Theorem A} can be regarded as an extension of \cite[Theorem 24]{BPRS} to the general case of martingales.

For $0<p<1$, if $f\in H^p(\Omega)$, $g\in \Lambda_1(\alpha_p)$ and $f_0=g_0=0$, then their product can be regarded as a continuous linear functional on $L^\8(\Omega)\cap \Lambda_1(\alpha_p)$. To be more precise, for any $h\in L^\8(\Omega)\cap \Lambda_1(\alpha_p)$, define
$$ 
  \la f \times g, h\ra := \la h\cdot g, f\ra ,  
$$
where in the right-hand side the duality between $H^p(\Omega)$ and $\Lambda_1(\alpha_p)$ is invoked. Note that $h\cdot g$ belongs to $\Lambda_1(\alpha_p)$ since $h$ is a pointwise multiplier on $\Lambda_1(\alpha_p)$ (see \cite{NS}).

Our following theorem establishes a continuous bilinear decomposition for products between  elements in $H^p(\Omega)$ and functions in the dual space $\Lambda_1(\alpha_p)$ when $0<p<1$.

\begin{thm}\label{Theorem C}
Let $(\Omega, \F, P)$ be a probability space equipped with the filtration $\{\F_k\}_{k \ge 1}$, where $\F_k$ is generated by countably many atoms for any $k \ge 1$. 

If $H^p(\Omega)\ (0<p<1)$ are martingale Hardy spaces, then there exist continuous bilinear operators $\Pi_1:H^p(\Omega)\times \Lambda_1(\alpha_p) \to L^1(\Omega),\ \Pi_2:H^p(\Omega)\times \Lambda_1(\alpha_p) \to H^1(\Omega)$ and $ \Pi_3:H^p(\Omega)\times \Lambda_1(\alpha_p) \to H^{p}(\Omega) $   such that
\begin{equation*}
f \times g = \Pi_1(f,g)+\Pi_2(f,g) + \Pi_3(f,g) 
\end{equation*}
for all $ f\in H^p(\Omega) $ and $ g\in \Lambda_1(\alpha_p)$.
\end{thm}

In the second part of this paper, we study analogues of Theorems \ref{Theorem A} and \ref{Theorem C} for the case of dyadic martingales on spaces of homogeneous type. Such martingales were first constructed in \cite{Hy}. We investigate the corresponding martingale Hardy spaces and extend Mei's results in \cite{MT} to this general setting. 
Compared with the probability setting, the case of spaces of homogeneous type is more difficult to deal with since backward martingales arise, and the underlying measures on homogeneous spaces may be infinite.

The present paper is organized as follows. In section \ref{Preliminaries}, we set down notation and give some background on martingale Hardy--Orlicz spaces. 
In section \ref{case1}, we prove Theorem \ref{Theorem A}. In section \ref{case2}, we present a characterization of martingale Lipschitz spaces $\Lambda_{1}(\alpha_p)$, which is of independent interest  (see Theorem \ref{thm4.5} and Remark \ref{rmk_Lip} below), and then we show Theorem \ref{Theorem C}. The remaining sections are concerned with spaces of homogeneous type. For the convenience of the reader, in section \ref{case3}, we recall some definitions and facts regarding Hardy spaces and Lipschitz spaces on spaces of homogeneous type in the sense of Coifman and Weiss \cite{Co}. In section \ref{section6}, we give some new proofs of results in \cite{Co} based on martingale methods and the existence of dyadic martingales on homogeneous spaces. In section \ref{section7}, we establish analogues of Theorems \ref{Theorem A} and  \ref{Theorem C}  for dyadic martingales on spaces of homogeneous type; see Theorem \ref{thm7.5} below. In the last section, we apply Theorem \ref{thm7.5} to obtain a decomposition of products of functions in Hardy spaces and their dual spaces on spaces of homogeneous type.

\section{Notation and Background}\label{Preliminaries}

In this section, we provide some notation and background that will be used in this paper. 

\subsection{Notation}\label{notation}

In several parts of this paper, we  consider sums and intersections of quasi-normed vector spaces. For the convenience of the reader we recall these notions below.
 
\begin{definition}
Let $(X_1, \|\cdot\|_{X_1} ), (X_2, \|\cdot\|_{X_2})$ be two quasi-normed vector spaces and let $X$ be a topological vector space $X$ such that $X_1, X_2 \subset X$ continuously.
	
\begin{enumerate}
\item  $(X_1 \cap X_2, \|\cdot\|_{X_1 \cap x_2})$ is the intersection of $X_1$ and $X_2$, where 
$$\| x \|_{X_1 \cap X_2 } := \max\{\|x\|_{X_1}, \|x\|_{X_2}\}$$ 
for all $x\in X_1 \cap X_2$;
\item $(X_1+X_2, \|\cdot\|_{X_1+X_2})$ is the sum of $X_1$ and $X_2$, where 
$$\|x\|_{X_1+X_2}:=\inf\{\|x_1\|_{X_1}+\|x_2\|_{X_2} : \ x=x_1+x_2, \ x_1\in X_1, \ x_2\in X_2 \} $$
for all $x\in X_1 + X_2$.
\end{enumerate}

For convenience, the  sum $ X_1 + X_2 + \cdots + X_n $ and the intersection $ X_1 \cap X_2 + \cdots \cap X_n $  will also be denoted by $ \sum\limits_{k = 1}^{n} X_k $  and $ \bigcap_{k = 1}^{n} X_k $, respectively.

Note that  $(X_1 \cap X_2, \|\cdot\|_{X_1 \cap X_2 })$ and $(X_1+X_2, \|\cdot\|_{X_1+X_2})$ are both quasi-normed vector spaces. Moreover, if $(X_1, \|\cdot\|_{X_1})$ and $(X_2, \|\cdot\|_{X_2})$ are Banach spaces, then $(X_1 \cap X_2, \|\cdot\|_{X_1 \cap X_2})$ and $(X_1+X_2, \|\cdot\|_{X_1+X_2})$ are both Banach spaces.
\end{definition}

In this article we shall use the following standard notation: $A\lesssim B$ (resp. $A\lesssim_p B$) means that $A\le C B$ (resp. $A\le C_p B$) for some absolute positive constant $C$ (resp. a positive constant $C_p$ depending only on a parameter $p$). If  $A\lesssim B$ and $B \lesssim A$ (resp. $A \lesssim_p B$ and $B \lesssim_p A$), we write $A\approx B$ (resp. $A\approx_p B$).

Throughout the paper, the terms ``homogeneous spaces'' and ``spaces of homogeneous type'' will be used interchangeably.

\subsection{Musielak--Orlicz-type spaces}\label{M-O} We shall first recall some definitions and properties of Orlicz-type spaces and Musielak--Orlicz-type spaces. In what follows, $(\Omega, \F, \mu) $ denotes a $\sigma$-finite measure space.

A function $ \Phi: [0,\infty) \to [0,\infty) $ is called an Orlicz function if it is strictly positive on $(0,\infty)$, non-decreasing, unbounded and $\Phi(0)=0$. A measurable function $ \Psi: \Omega\times [0,\infty) \to [0,\infty) $ is called a Musielak--Orlicz function if for all $ x\in \Omega $, $ \Psi(x,\cdot) $ is an Orlicz function.

The Musielak--Orlicz-type space $L^{\Psi}(\Omega)$ is the set consisting of all measurable functions $f$ on $\Omega$ such that
$$
\int_{\Omega}\Psi(x,\lambda^{-1} |f(x)|) d\mu < \infty
$$
for some $ \lambda >0 $. We equip $L^{\Psi}(\Omega)$ with the Luxembourg quasi-norm
$$
\| f \|_{L^{\Psi}(\Omega)} := \inf\lk{\lambda >0 : \int_{\Omega}\Psi(x,\lambda^{-1} |f(x)|) d\mu \leqslant 1 }, \quad f \in L^{\Psi}(\Omega).
$$

Let $p \in \mathbb{R}$. A Musielak--Orlicz function is said to be of uniformly lower type (respectively, upper type) $p$ if there exists a positive constant $C$ such that
$$
\Psi(x,st) \leqslant Cs^p\Psi(x,t) 
$$
for all $ x\in \Omega$, $t\geqslant 0$ and $s\in (0,1)$ (respectively, $s\in[1,\infty)$).
In particular, if $\Psi$ is of uniformly lower type $p$ with $0<p<1$ and of uniformly upper type $1$ then
\begin{equation}\label{Psic}
\Psi(x,ct) \approx_c \Psi(x,t) \quad \text{for all } c>0.
\end{equation}

In the sequel, $\Psi(x,t)$ is always assumed to be of uniformly lower type $p$ with $0<p<1$ and of uniformly upper type $1$, and to be continuous in the $ t $ variable. For more information on Musielak--Orlicz spaces, we refer the reader to \cite{Bo} and \cite{YLK}.

\subsection{Martingales}\label{martingales}Let $(\Omega, \F, P)$ be a fixed probability
space. Given a filtration which consists of a sequence of $\sigma$-algebras
$$ \F_1\subset\dots\subset \F_k\subset \dots \subset \F $$
such that $\sigma(\cup_{k=1}^{\infty}\F_k)=\F$, for a random variable $f\in L^{1}(\Omega, \F, P)$ and $k \in \mathbb{N}_+ $, we set 
$$ f_k =\mathbb{E}\left(f \mid{\F}_{k}\right), \quad  d_k f  = f_k-f_{k-1}, $$
where we adopt the convention that $f_{0}=0$. We shall also denote $f_k $ by $\ce_{k}(f)$. The sequence $\{f_k\}_{k\geq 0}$ is called the martingale of $f$, and $ \{ d_k f \}_{k \geq 1} $ is called the martingale difference of $f$. If $f$ and $ \{f_k\}_{k\geq 0}$ are as above, we shall also write $f =  \{f_k\}_{k\geq 0}$. To simplify notation, we write $L^p(\Omega)$  instead of $ L^{p}(\Omega, \F, P)$, $0<p<\8$.

\begin{definition} If $f$, $\{f_k\}_{k\geq 0}$ and  $\{d_k f\}_{k\geq 1}$ are as above, we define:
\begin{enumerate}
\item the maximal function
$$ f^*:= \sup\limits_{k\geqslant 0} |f_k|; $$
\item the square function
$$ S(f):= \sk{ \sum_{k=1}^{\infty} |d_kf|^2 }^{\frac{1}{2}}; $$
\item the conditional square function
$$ s(f):= \sk{ \sum_{k=1}^{\infty} \ce_{k-1} |d_kf|^2 }^{\frac{1}{2}}. $$
\end{enumerate}
\end{definition}

There are several types of martingale Hardy spaces, which are defined in terms of maximal functions, square functions and conditional square functions.
\begin{definition}\label{hardy}
 For $ 1\leq p<\8$, the martingale Hardy spaces $h^p(\Omega), H^p(\Omega), H^p_m(\Omega)$ are defined as follows
 \begin{align*}
 h^p(\Omega)& :=\{f\in L^1(\Omega) : \| f \|_{h^p} := \| s(f) \|_p < \infty\}, \\
 H^p(\Omega)& : =\{f\in L^1(\Omega) :  \| f \|_{H^p} := \| S(f) \|_p < \infty\}, \\
 H^p_m(\Omega)&: =\{f\in L^1(\Omega) : \| f \|_{H^p_m} := \| f^* \|_p < \infty\},
 \end{align*}
 respectively.
  
For $0<p<1$, $h^p(\Omega)$ is defined as the completion of the space $\{f\in L^1(\Omega) : \| f \|_{h^p} := \| s(f) \|_p < \infty\}$ with respect to the norm $\| \cdot\|_{h^p}$. Similarly, $H^p(\Omega)$ is defined as the completion of the space $\{f\in L^1(\Omega) : \| f \|_{H^p} := \| S(f) \|_p < \infty\}$ with respect to the norm $\| \cdot\|_{H^p}$, and $H^p_m(\Omega)$ is defined as the completion of the space $\{f\in L^1(\Omega) : \| f \|_{H_m^p} := \| f^* \|_p < \infty\}$ with respect to the norm $\| \cdot \|_{H_m^p}$.
\end{definition}

In general, the above three martingale Hardy spaces are different. However, for $1\leq p<\8 $, $H^p(\Omega)=H^p_m(\Omega)$ (see \cite{BD}, \cite{DB}, \cite{WF2}).

\begin{definition}\label{reg_filt}(Regular filtration)
A filtration is regular if there exists a constant $C>0$ such that for all $k \geq 2, F_{k} \in \F_{k}$, there exists a $G_{k} \in \F_{k-1}$ satisfying
$$
F_{k} \subset G_{k}, \quad P(G_{k})\leq C \cdot P(F_{k}).
$$
In addition, a martingale $f=\{f_k\}_{k\geq 0}$ with respect to such a regular filtration is called a regular martingale.
\end{definition}

\begin{remark}  Suppose that for a positive random variable $ f \in L^1(\Omega)$ the corresponding martingale $\{ f_k \}_{k \geqslant 0}$ is regular. Then  for any $ k \geq 2$ 
\begin{equation*}\label{regm}
f_k \leqslant A \cdot f_{k-1},
\end{equation*} 
where $A>0$ is a constant that depends only on the constant $C$ of Definition \ref{reg_filt}.

See \cite{Long} for more information about regular filtrations and martingales.
\end{remark}

\begin{remark}
For regular martingale filtrations, $ H^p(\Omega) = h^p(\Omega) = H^p_m(\Omega) $ when $0<p<\8$. See \cite{WF1}, \cite{WF2} and \cite{Long} for more information.
\end{remark}

An important aspect of martingale Hardy spaces is that they admit atomic decompositions. The definition of atoms in the martingale setting is given below.

\begin{definition} A random variable $ a : \Omega \to \com $ is called a martingale simple $(p,q)$-atom ($ 0<p\leq 1, 1\leq q\leqslant \infty $) if there exist $k \in \nat$ and $ A \in \F_k $ such that
\begin{enumerate}
\item $ \ce_k(a) = 0 $;
\item $ \supp (a) \subset A $;
\item $ \| a \|_q \leqslant P(A)^{\frac{1}{q}-\frac{1}{p}}$,
\end{enumerate}
where $\frac{1}{q}:=0$ when $q=\8$ as convention.
\end{definition}

\begin{definition}\label{ath}
We define the martingale atomic Hardy spaces $ H_{\rm at}^{p,q}(\Omega) $ for $ 0 < p < 1 \leqslant q \leqslant \infty $ or $p=1, \ 1<q\leq \8$ as follows
$$
H_{\rm at}^{p,q}(\Omega) := \lk{f = \sum_{j=0}^{\infty} \lambda_j a^j  \text{  where $a^j$ is a simple $(p,q)$-atom and } \sum_{j=0}^\8|\lambda_j|^p < \infty },
$$
where for $f\in H_{\rm at}^{p,q}(\Omega)$
$$ \|f\|_{H^{p,q}_{\rm at}(\Omega)}:=\inf\lk{ \big(\sum_{j= 0}^\8|\lambda_j|^p\big)^\frac{1}{p} : f = \sum_{j=0}^{\infty} \lambda_j a^j, \text{ where $a^j$ is a simple $(p,q)$-atom} }. $$
\end{definition}

It is well-known that $h^p(\Omega)=H_{\rm at}^{p,2}(\Omega)$ when $0<p\leq1$ (see \cite{WF1}). In particular, if the martingale filtration is regular, then $h^p(\Omega)=H_{\rm at}^{p,q}(\Omega)$ when $0<p\leq1$ and $1<q\leq \8$. The following result is the atomic decomposition of $H^1(\Omega)$, which follows from the noncommutative result in \cite{PM}. In particular, it reveals the relationship between $H^1(\Omega)$ and $h^1(\Omega)$  and shows that $H^1(\Omega)\neq h^1(\Omega)$ for general martingales.

\begin{thm}\label{dec of H}
We have $H^1(\Omega)=h^1(\Omega)+h^1_d(\Omega)$, where $h^1_d$ denotes the diagonal Hardy space of martingale differences
$$ h^1_d(\Omega):=\lk{h\in L^1(\Omega) : \| h \|_{h^1_d(\Omega)} := \sum_{k= 1}^\8 \| d_kh \|_1 < \infty}. $$
\end{thm}

We shall now introduce the martingale $BMO$ and $bmo$ spaces, which are the duals of $H^1(\Omega)$ and $h^1(\Omega)$, respectively (see Theorem \ref{duality} below).

\begin{definition} Assume $f, g \in L^2(\Omega)$. We say that $f$ is a martingale $BMO$ function if
$$ \| f \|_{BMO(\Omega)} := \sup_{n\geqslant 1}\| \ce_{n}|f-f_{n-1}|^2 \|_{\infty}^{\frac{1}{2}} < \infty. $$
We say that $g$ is a martingale $bmo$ function  if
$$ \| g \|_{bmo(\Omega)} := \sup_{n\geqslant 0}\| \ce_{n}|g-g_n|^2 \|_{\infty}^{\frac{1}{2}} < \infty. $$

Denote by $BMO(\Omega)$ and $bmo(\Omega)$ the spaces consisting of all martingale $BMO$ and $bmo$ functions, respectively.
\end{definition}

For regular martingales, $BMO(\Omega)=bmo(\Omega)$. The following result is the so-called martingale John--Nirenberg inequality and can be found in \cite{Ga}.

\begin{thm}\label{MJN}
There exists a sufficiently small constant $\kappa>0$ such that for any $f\in BMO(\Omega)$ with $\|f\|_{BMO(\Omega)}\leq \kappa$, we have
$$ \ce\sk{ e^{|f|} } \leqslant 1. $$
\end{thm}

\begin{rem} From the martingale John--Nirenberg inequality, we have for any $1\leq p<\8$,
$$  \| f \|_{BMO(\Omega)} \approx_p \sup_{n\geqslant 1}\| \ce_{n}|f-f_{n-1}|^p \|_{\infty}^{\frac{1}{p}} . $$
However, the above John--Nirenberg inequality fails for $bmo(\Omega)$ in the general martingale setting. 
\end{rem}
 
For the following duality theorem, see \cite{Ga}, \cite{Long}, \cite{WF2}.
\begin{thm}\label{duality}
  $ (H^1(\Omega))^* = BMO(\Omega)$ and $\ (h^1(\Omega))^* = bmo(\Omega) $.
\end{thm}

The following proposition, which can be found in \cite{Al} and \cite{Ga}, is a consequence of Theorems \ref{dec of H} and \ref{duality} and it gives a description of the relationship between $BMO(\Omega)$ and $bmo(\Omega)$. In particular, it implies that $BMO(\Omega)\subsetneqq bmo(\Omega)$ for general martingales.

\begin{proposition}\label{Bb} Assume $f$ is a martingale $ BMO $ function. Then
\begin{equation}
\| f \|_{BMO(\Omega)} \approx \| f \|_{bmo(\Omega)} + \sup\limits_{k\geq 1} \| d_kf \|_{\infty}.
\end{equation}
\end{proposition}

We end this section with the definition of martingale Musielak--Orlicz Hardy spaces and the generalized H\"{o}lder inequality.
\begin{definition}\label{def_M-H-O}
The martingale Musielak--Orlicz Hardy space $H^{\Psi}(\Omega)$ (where $\Psi$ is described in Section \ref{M-O}) is the space consisting of all martingales $f=\{f_k\}_{k\geq 0}$ such that the square function $S(f) \in L^{\Psi}(\Omega)$. Moreover, we define the quasi-norm
$$ \|f\|_{H^\Psi(\Omega)}:=\|S(f)\|_{L^\Psi(\Omega)}. $$
If $\Psi$ is replaced by an Orlicz function $\Phi$, the corresponding Hardy--Orlicz space $H^{\Phi}(\Omega)$ is defined in an analogous way.
\end{definition}

To obtain the generalized H\"{o}lder inequality, we shall introduce a particular Orlicz space $L^\Phi(\Omega)$, where
\begin{equation}\label{Phi}
 \Phi(t): = \frac{t}{\log(e+t)}, \quad t\geq 0.
 \end{equation}
Note that $\Phi$ is an Orlicz function of uniformly lower type $p \ (0<p<1)$ and upper type $1$, which guarantees that the vector space $L^\Phi(\Omega)$ is a quasi-normed space. Note that $L^1(\Omega)\subset L^\Phi(\Omega)$.

\begin{remark}
It follows from \cite{Mi} that if $f=\{f_k\}_{k\geq 0}$ is a regular martingale, then the martingale Hardy--Orlicz space $H^{\Phi}(\Omega)$ can also be characterized by martingale maximal functions and conditional square functions. For any $f\in H^\Phi(\Omega)$ one has
$$ \|f\|_{H^\Phi(\Omega)}= \| S(f)\|_{L^\Phi(\Omega)}\approx \| f^*\|_{L^\Phi(\Omega)}\approx \| s(f)\|_{L^\Phi(\Omega)}. $$
\end{remark}

The following lemma is a variant of \cite[Proposition 2.1]{Bo1} in the martingale setting.

\begin{lemma}\label{HolderP}
Assume $ (\Omega, \F , P) $ is a probability space, $f \in L^1(\Omega)$ and $ g \in BMO(\Omega) $. Then $f\cdot g\in L^{\Phi}(\Omega)$ and
\begin{equation}\label{2.4}
  \| f\cdot g \|_{L^{\Phi}(\Omega)} \lesssim \| f \|_{1}\| g \|_{BMO(\Omega)}.
\end{equation}

\end{lemma}

\begin{proof}

The proof is similar to the proof of the corresponding Euclidean result and we shall only outline it here for the convenience of the reader.  
By  \cite[Lemma 2.1]{Bo1}, one has
\begin{equation}\label{Mst}
\dfrac{st}{ M + \log(e+st) } \leq   e^{t-M} + s .
\end{equation}
for all $M\geq 0, s\geq 0, t\geq 0$.

When $\|f\|_1=0$ or $\|g\|_{BMO(\Omega)}=0$,  $(\ref{2.4})$ trivially holds. Assume $g\in BMO(\Omega)$ with $\|g\|_{BMO(\Omega)}>0$ and $f\in L^1(\Omega)$ with $\|f\|_1>0$. Let $ \kappa $ be the constant in Theorem \ref{MJN}, $M=0$, $t=\frac{ \kappa |g(x)| }{ \|g\|_{BMO(\Omega)} }$ and $s=\frac{|f(x)|}{\|f\|_1}$. Then by Theorem \ref{MJN} and \eqref{Mst}, we have
\begin{align}\label{2.5}
\int_{\Omega} \Phi\sk{ \frac{ |f(x)\cdot g(x)|}{ \kappa^{-1}\| f \|_1 \| g \|_{BMO(\Omega)} }  } dP & \leq  \int_{\Omega} e^{ \frac{ \kappa |g(x)| }{ \|g\|_{BMO(\Omega)} } }dP +  \norm{\frac{f}{\|f\|_1}}_1 \leq 2.
\end{align}
Hence, from (\ref{Psic}) we conclude
$$   \| f\cdot g \|_{L^{\Phi}(\Omega)} \lesssim \kappa^{-1}\| f \|_{1}\| g \|_{BMO(\Omega)}, $$
which completes the proof of the lemma.
\end{proof}

We shall refer to \eqref{2.4} as the generalized H\"{o}lder inequality.

\section{Bilinear decompositions on $H^1(\Omega)\times BMO(\Omega)$}\label{case1}
In this section we prove Theorem \ref{A}. Let $ (\Omega, \F, P) $ be a fixed probability space and let $f\in H^1(\Omega), \ g\in BMO(\Omega)$. If we assume that $f $ and $g$ have finite martingale expansions, then we may write their pointwise product $f \cdot g$ as follows
\begin{equation}\label{dec}
  f\cdot g  = \Pi_1(f,g) + \Pi_2(f,g) + \Pi_3(f,g), 
\end{equation}
where
$$ \Pi_1(f,g) : =  \sum_{k=1}^{\infty} d_k f d_k g, \quad \Pi_2(f,g) : = \sum_{k=1}^{\infty} f_{k-1}d_k g  \quad \text{and} \quad \Pi_3(f,g) : =  \sum_{k=1}^{\infty} g_{k-1}d_k f. $$

We shall estimate $\Pi_1(f,g), \Pi_2(f,g)$, $\Pi_3(f,g)$ separately. To do so, we shall make use of the atomic decomposition of $H^1(\Omega)$. It follows from our arguments below that the operators $\Pi_1 $, $ \Pi_2 $ and $\Pi_3 $ are well-defined (in a pointwise sense) on the product space $H^1(\Omega) \times BMO(\Omega)$. Hence, the proof of Theorem will follow from the boundedness properties of  $\Pi_1 $, $ \Pi_2 $ and $\Pi_3 $, \eqref{dec} and a limiting argument.

In \S \ref{new_proof_Pi_3}, we present a direct way to deal with $\Pi_3(f,g)$, which avoids the use of the atomic decomposition.

\begin{proof}[Proof of Theorem \ref{Theorem A}]
By Theorem \ref{dec of H}, there always exist two functions $ f^h $ and $ f^d $ such that $ f= f^h + f^d $, where $ f^h \in h^1(\Omega) $ and $ f^d \in h^1_d(\Omega) $. For any such decomposition of $f$, since $f^h\in h^1(\Omega)$, there exist $ \lk{\lambda_j}_{j\geqslant 1} \subset \real$ and simple $(1,2)$-atoms $ \lk{a^j}_{j \geqslant 1} $ such that
\begin{equation}\label{atom_dec}
 f^h = \sum_{j=1}^{\infty} \lambda_j a^j,\quad \| f^h \|_{h^1(\Omega)} \approx \sum_{j=1}^{\infty} |\lambda_j|,
\end{equation}
where we assume $\supp (a^j) \subset A_{n_j} $ and $ A_{n_j} \in \F_{n_j} $ with $P(A_{n_j})>0$ for $j\geq 1$. Then
\begin{equation}\label{dec of pi}
\Pi_i(f,g) = \sum_{j=1}^{\infty}\lambda_j \Pi_i(a^j,g) + \Pi_i(f^d,g),\quad i = 1,2,3.
\end{equation}

\subsection{Estimates for $ \Pi_1(f^h,g) $ and $ \Pi_1(f^d,g) $}\label{Pi_1_est}

We are going to show that $ \Pi_1 $ is a bounded bilinear operator from $ H^1(\Omega) \times BMO(\Omega) $ to $ L^1(\Omega) $. In fact, the boundedness of $\Pi_1$ follows naturally from the duality between $H^1(\Omega)$ and $BMO(\Omega)$, i.e. Theorem \ref{duality} (see \cite{Ga}).  For the reader's convenience, we give a proof based on atomic decompositions.

We first focus on $\Pi_1(f^h,g)$, which can further be decomposed into atoms as described in \eqref{atom_dec}. It thus suffices to consider
 $$\Pi_1(a^j,g) = \sum_{k=1}^{\infty} d_k a^j d_k g,  $$
 which can further be localized as $d_k a^j=1_{A_{n_j}}d_k a^j$ when $k\geq n_j+1$ since $A_{n_j}\in \F_{n_j}$, namely
$$
\Pi_1(a^j,g)  = \sum_{k=n_j+1}^{\infty} 1_{A_{n_j}} d_k a^j d_k g.
$$
Now, by applying the Cauchy-Schwarz inequality, we derive the estimate

\begin{align*} 
   \| \Pi_1(a^j,g) \|_1 & = \ce\sk{ \jdz{  \sum_{k=n_j+1}^{\infty} 1_{A_{n_j}} d_k a^j d_k g } }  \\
   & \leqslant \mk{ \ce \sk{ \sum_{k=n_j+1}^{\infty} |d_ka^j|^2 } }^{\frac{1}{2}} \mk{ \ce \sk{ \sum_{k=n_j+1}^{\infty}1_{A_{n_j}} |d_kg|^2 } }^{\frac{1}{2}}   \\
   & \leqslant \| a^j \|_2\mk{ \ce \ce_{n_j} \sk{ \sum_{k=n_j+1}^{\infty}1_{A_{n_j}} |d_kg|^2 } }^{\frac{1}{2}}  \\
   & \leqslant P(A_{n_j})^{-\frac{1}{2}}\mk{ \ce \sk{ 1_{A_{n_j}}\ce_{n_j} \sk{ \sum_{k=n_j+1}^{\infty} |d_kg|^2} } }^{\frac{1}{2}} \notag\\
   & \leqslant P(A_{n_j})^{-\frac{1}{2}}\| g \|_{bmo(\Omega)}P(A_{n_j})^{\frac{1}{2}} 
 \end{align*}
where the fourth inequalitiy follows from the definition of the atom. Hence, we deduce from the definition of the $bmo-$norm that 
\begin{equation}\label{pi-1aj}
\| \Pi_1(a^j,g) \|_1 \leq \| g \|_{bmo(\Omega)}. 
\end{equation}
By using (\ref{pi-1aj}) and (\ref{dec of pi}), we have by Theorem \ref{Bb}
\begin{align*}
  \| \Pi_1(f,g) \|_1 & \leqslant \sum_{j=1}^{\infty} |\lambda_j|\| g \|_{bmo(\Omega)} + \norm{ \sum_{k=1}^{\infty} d_kf^d d_k g }_1 \notag\\
  &\lesssim \| f^h \|_{h^1}\| g \|_{bmo(\Omega)} + \sk{\sup\limits_{k\geq 1} \| d_k g \|_{\infty}}\sk{\sum_{k=1}^{\infty} \| d_kf^d \|_1 } \notag\\
  & \lesssim \sk{\| f^h \|_{h^1(\Omega)} + \| f^d \|_{h_1^d(\Omega)}}  \| g \|_{BMO(\Omega)}.
\end{align*}
Since the decomposition of $ f = f^h + f^d $ is arbitrary, by Theorem \ref{dec of H} we conclude
\begin{equation}\label{pi-1}
  \| \Pi_1(f,g) \|_1 \lesssim \| f \|_{H^1(\Omega)} \| g \|_{BMO(\Omega)}.
\end{equation}

\subsection{Estimates for $ \Pi_2(f^h,g) $ and $ \Pi_2(f^d,g) $} We are going to show that $ \Pi_2 $ is a bounded bilinear operator from $ H^1(\Omega) \times BMO(\Omega) $ to $ H^1(\Omega) $. Arguing as in section \ref{Pi_1_est}, we perform the localization on each term
$$
\Pi_2(a^j,g) = \sum_{k=1}^{\infty} a^j_{k-1} d_k g = \sum_{k=n_j+2}^{\infty} 1_{A_{n_j}} a^j_{k-1} d_k g.
$$
It is easy to verify that
$$
d_k(\Pi_2(a^j,g)) = a^j_{k-1} d_k g, \quad k \geq n_j+2 \ \text{ and } \ d_k(\Pi_2(a^j,g)) = 0 ,\quad 1\leqslant k \leqslant n_j+1 .
$$
We consider the corresponding square function
\begin{align*}
S\sk{ \Pi_2(a^j,g) } & = \sk{ \sum_{k=n_j+2}^{\infty} \sk{|a^j_{k-1}|^21_{A_{n_j}} |d_k g|^2} }^{\frac{1}{2}} \\
& \leqslant |(a^j)^*|\sk{ \sum_{k=n_j+2}^{\infty} 1_{A_{n_j}} \sk{|d_k g|^2} }^{\frac{1}{2}}.
\end{align*}
Then by invoking the Cauchy-Schwarz inequality, we have that 
\begin{align*} 
  \| \Pi_2(a^j,g) \|_{H^1(\Omega)} & = \ce\mk{ S\sk{ \Pi_2(a^j,g) } }    \\
   & \leqslant \| (a^j)^* \|_2 \mk{ \ce\sk{ \sum_{k=n_j+2}^{\infty} 1_{A_{n_j}} (|d_k g|^2)}  }^{\frac{1}{2}}  \\
   & \leqslant \| a^j \|_2 \mk{ \ce \sk{1_{A_{n_j}} \ce_{n_j}\sk{\sum_{k=n_j+2}^{\infty} |d_k g|^2}} }^{\frac{1}{2}} \\
   & \leqslant P(A_{n_j})^{-\frac{1}{2}}\| g \|_{BMO(\Omega)}P(A_{n_j})^{\frac{1}{2}} \end{align*}
and hence,
\begin{equation}\label{pi-2}
 \| \Pi_2(a^j,g) \|_{H^1(\Omega)} \leqslant \| g \|_{BMO(\Omega)}.
\end{equation}\label{pi-2}

Similarly, by Theorem \ref{Bb}
\begin{align*} 
\norm{ \Pi_2(f^d,g) }_{H^1(\Omega)} & =  \ce\mk{ S\sk{ \Pi_2(\sum_{m=1}^{\infty}d_mf^d,g) } } \\
 & \leqslant \sum_{m=1}^{\infty}\ce\mk{ S\sk{ \Pi_2(d_mf^d,g) } }  \\
 & = \sum_{m=1}^{\infty}\ce\mk{ \ce_m\sk{ \sum_{k=m+1}^{\infty} |d_mf^d|^2 |d_k g|^2 }^{\frac12}}   \\
 & =  \sum_{m=1}^{\infty}\ce\mk{ |d_mf^d| \ce_m\sk{ \sum_{k=m+1}^{\infty} |d_k g|^2 }^{\frac12} }  \\
   & \leqslant  \sum_{m=1}^{\infty}\sk{\| d_mf^d \|_1   \norm{ \ce_m \sk{\sum_{k=m+1}^{\infty} |d_k g|^2} }_{\infty}^{\frac12} } 
\end{align*}
and hence,
\begin{equation}\label{pi-2fd}
\norm{ \Pi_2(f^d,g) }_{H^1(\Omega)}   \leqslant \| f^d \|_{h^1_d(\Omega)} \|g\|_{BMO(\Omega)}.
\end{equation}
By using \eqref{pi-2fd}, \eqref{pi-2} and \eqref{dec of pi}, we have by Theorem \ref{dec of H}
\begin{align*}
\| \Pi_2(f,g) \|_{H^1(\Omega)} &\leq \norm{ \Pi_2(f^h,g) }_{H^1(\Omega)}+\norm{ \Pi_2(f^d,g) }_{H^1(\Omega)}\notag\\
&\leqslant \sum_{j=1}^{\infty}|\lambda_j|\| g \|_{BMO(\Omega)} + \| f^d \|_{h_1^d(\Omega)}\|g\|_{BMO(\Omega)}\\ &\lesssim \sk{\| f^h \|_{h^1(\Omega)} + \| f^d \|_{h^1_d(\Omega)} } \| g \|_{BMO(\Omega)}.
\end{align*}
Since the decomposition of $ f = f^h + f^d $ is arbitrary, by Theorem \ref{dec of H} we conclude
\begin{equation}\label{pi--2}
  \| \Pi_2(f,g) \|_{H^1(\Omega)} \lesssim \| f \|_{H^1(\Omega)}\|g\|_{BMO(\Omega)}.
\end{equation}

\subsection{Estimates for $ \Pi_3(f^h,g) $ and $ \Pi_3(f^d,g) $}\label{pi}
We are going to show that $ \Pi_3 $ is a bounded bilinear operator from $ H^1(\Omega) \times BMO(\Omega) $ to $ H^{\Phi}(\Omega) $. To this end, we first deal with
$ \Pi_3(f^h,g) $. Note that
\begin{align*}
S(\Pi_3(f^h,g)) & = S\sk{ \sum_{k=1}^{\infty} \sum_{j=1}^{\infty} \lambda_jg_{k-1} d_ka^j }
  \leq  \sum_{j=1}^{\infty} \lambda_jS\sk{\sum_{k=1}^{\infty}g_{k-1} d_ka^j}  \\
 & = \sum_{j=1}^{\infty}\lambda_j\sk{ \sum_{k=n_j+1}^{\infty} |g_{k-1}|^2 |d_ka^j|^2 }^{\frac{1}{2}}  \\
& \leq \sum_{j=1}^{\infty}\lambda_j\sk{ \sum_{k=n_j+1}^{\infty} |g_{k-1}-g_{n_j}|^2| d_ka^j|^2 }^{\frac{1}{2}} + \sum_{j=1}^{\infty}\lambda_j|g_{n_j}|S( a^j)  \\
& =: I_1 + I_2 .
\end{align*}
It thus suffices to handle $I_1$ and $I_2$. For $I_1$, we have
\begin{align*} 
\ce(I_1) & \leqslant \sum_{j=1}^{\infty}|\lambda_j| \ce \sk{ \sum_{k=n_j+1}^{\infty} 1_{A_{n_j}} |g_{k-1}-g_{n_j}|^2 |d_ka^j|^2 }^{\frac{1}{2}}  \\
 & \leqslant \sum_{j=1}^{\infty}|\lambda_j| \mk{\ce \sk{\sum_{k=n_j+1}^{\infty} 1_{A_{n_j}}|g_{k-1}-g|^2 |d_ka^j|^2}^{\frac{1}{2}} + \ce \sk{\sum_{k=n_j+1}^{\infty} 1_{A_{n_j}}|g-g_{n_j}|^2 |d_ka^j|^2}^{\frac{1}{2}} } \\
&\leqslant \sum_{j=1}^{\infty}|\lambda_j|\lk{ P(A_{n_j})^{\frac{1}{2}} \mk{\ce \sk{\sum_{k=n_j+1}^{\infty} |g_{k-1}-g|^2 |d_ka^j|^2}}^{\frac{1}{2}} + \ce \sk{1_{A_{n_j}}|g-g_{n_j}|S(a^j)}  }  \\
 & \leqslant \sum_{j=1}^{\infty}|\lambda_j| \lk{ P(A_{n_j})^{\frac{1}{2}} \mk{\ce \sk{\sum_{k=n_j+1}^{\infty}|d_ka^j|^2 \ce_{k} (|g_{k-1}-g|^2) } }^{\frac{1}{2}} + \| a^j \|_2P(A_{n_j})^{\frac{1}{2}}\| g \|_{BMO(\Omega)} }  \\
 & \leqslant 2\sum_{j=1}^{\infty}|\lambda_j| P(A_{n_j})^{\frac{1}{2}}\| g \|_{BMO(\Omega)}\| a^j \|_2 
 \end{align*}
and so,
\begin{equation}\label{Pi-3-1}
 \ce(I_1) \lesssim \| f^h \|_{h^1(\Omega)}\| g \|_{BMO(\Omega)}.
 \end{equation}
Next, we obtain an estimate for $I_2$. To this end, notice that
\begin{align*}
  I_2  &\leq \sk{\sum_{j=1}^{\infty} 1_{A_{n_j}}|\lambda_j| S(a^j)}\cdot |g| + \sum_{j=1}^{\infty} |\lambda_j|1_{A_{n_j}}|g_{n_j}-g| S(a^j) \\
   & =: I_3 + I_4.
\end{align*}
Since $a^j$ is a simple $(1,2)$-atom, we have $ \|1_{A_{n_j}}S(a^j)\|_1\leq 1$ and
$$\norm{ \sum\limits_{j=1}^{\infty}1_{A_{n_j}} |\lambda_j| S(a^j)}_1\leq \sum\limits_{j=1}^{\infty} |\lambda_j| \lesssim  \| f^h \|_{h^1(\Omega)}. $$
By Lemma \ref{HolderP}, we have
\begin{equation}\label{i-3}
\| I_3 \|_{L^{\Phi}(\Omega)} \lesssim \norm{ \sum_{j=1}^{\infty}1_{A_{n_j}} |\lambda_j| S(a^j) }_1\| g \|_{BMO(\Omega)} \lesssim \| f^h \|_{h^1(\Omega)}\| g \|_{BMO(\Omega)}.
\end{equation}

The following estimate is implicit in the proof of \eqref{Pi-3-1}:
\begin{equation}\label{i-4}
  \ce(I_4) \leqslant \sum_{j=1}^{\infty}|\lambda_j| P(A_{n_j})^{\frac{1}{2}}\| g \|_{BMO(\Omega)}\| a^j \|_2 \lesssim \| f^h \|_{h^1(\Omega)}\| g \|_{BMO(\Omega)}.
\end{equation}

By combining \eqref{i-3} and \eqref{i-4}, we deduce that
\begin{equation}\label{Pi-3-2}
  \| I_2 \|_{L^{\Phi}(\Omega)} \lesssim \| f^h \|_{h^1(\Omega)} \| g \|_{BMO(\Omega)}.
\end{equation}
In conclusion, by (\ref{Pi-3-1}) and (\ref{Pi-3-2}) we get
\begin{equation}\label{Pi-3-fh}
	\|\Pi_3(f^h,g)\|_{H^\Phi(\Omega)} \lesssim \| f^h \|_{h^1(\Omega)} \| g \|_{BMO(\Omega)}.
\end{equation}

It remains to deal with $ \Pi_3(f^d,g) $. We have
\begin{align*}
S(\Pi_3(f^d,g)) & = \sk{ \sum_{k=1}^{\infty}|g_{k-1}|^2|d_kf^d|^2  }^{\frac{1}{2}} \leqslant \sum_{k=1}^{\infty}|g_{k-1}||d_kf^d| \\
& \leq \sum_{k=1}^{\infty}|g_{k-1}-g||d_kf^d| + |g|\sk{\sum_{k=1}^{\infty}|d_kf^d|}.
\end{align*}
By Lemma \ref{HolderP},
\begin{equation}\label{i-6}
  \norm{ g\sk{\sum_{k=1}^{\infty}|d_kf^d|} }_{L^{\Phi}(\Omega)} \lesssim \sk{ \sum_{k=1}^{\infty} \| d_kf^d \|_1 }\| g \|_{BMO(\Omega)}=\|f^d\|_{h^1_d(\Omega)}\| g \|_{BMO(\Omega)}.
\end{equation}

For the remaining term, we have
\begin{align*}\label{i-5}
  \ce \sk{ \sum_{k=1}^{\infty}|g_{k-1}-g||d_kf^d| } & = \ce \sk{ \sum_{k=1}^{\infty}|d_kf^d|\ce_{k}|g_{k-1}-g| }  \\
   & \leqslant \| g \|_{BMO(\Omega)}\sk{\sum_{k=1}^{\infty}\| d_kf^d \|_1} 
\end{align*}
and so
\begin{equation}\label{i-5}
  \ce \sk{ \sum_{k=1}^{\infty}|g_{k-1}-g||d_kf^d| }  
 \leqslant \|f^d\|_{h^1_d(\Omega)}\| g \|_{BMO(\Omega)}.
\end{equation}
Hence, by \eqref{i-6} and \eqref{i-5}, we get
\begin{equation}\label{Pi-3-fd}
\| \Pi_3(f^d,g) \|_{H^{\Phi}(\Omega)} \lesssim\|f^d\|_{h^1_d(\Omega)}\| g \|_{BMO(\Omega)}.
\end{equation}
By \eqref{Pi-3-fh} and \eqref{Pi-3-fd}, we obtain
$$
\| \Pi_3(f,g) \|_{H^{\Phi}(\Omega)} \lesssim \sk{\|f^h\|_{h^1(\Omega)} + \| f^d \|_{h^1_d(\Omega)}}\| g \|_{BMO(\Omega)}.
$$
Thus we conclude
\begin{equation}\label{pi-3}
\| \Pi_3(f,g) \|_{H^{\Phi}(\Omega)} \lesssim \| f \|_{H^1(\Omega)}\| g \|_{BMO(\Omega)}.
\end{equation}
This completes the proof of Theorem \ref{Theorem A}
\end{proof}

\subsection{Another method for handling $\Pi_3(f,g)$}\label{new_proof_Pi_3}

In this section we present a different method for dealing with $\Pi_3(f,g)$, which is much neater and simpler than the one presented above, and it relies on the following theorem which has been shown in \cite{Ga}.
\begin{thm}\label{mBMO}
  If $ g \in BMO(\Omega) $ and $g_0=0$, then  $(g^*)_0\lesssim \|g\|_{BMO(\Omega)}$ and $ g^* \in BMO(\Omega) $.  Moreover, $ \| g^* \|_{BMO(\Omega)} \lesssim \| g \|_{BMO(\Omega)}. $
\end{thm}

We begin with a pointwise estimate for $S ( \Pi_3(f,g))$. Towards this aim, note that $ d_k(\Pi_3(f,g)) = g_{k-1}d_kf $, which implies that
  $$ S(\Pi_3(f,g)) = \sk{ \sum_{k=1}^{\infty} |g_{k-1}|^2|d_kf|^2 }^{\frac12}
  \leqslant |g^*|S(f) \leqslant J_1+J_2, $$
 where
 $$  J_1 :=  |g^*-(g^*)_0|S(f) \quad \text{and} \quad  J_2 := S(f)\|g\|_{BMO(\Omega)}. $$
 Clearly,
\begin{equation}\label{ms2}
\|J_2\|_1\lesssim \| f \|_{H^1(\Omega)} \| g \|_{BMO(\Omega)}.
\end{equation}

By Theorem \ref{mBMO}, we get $ g^* \in BMO(\Omega) $, and hence by Lemma \ref{HolderP}
\begin{equation}\label{ms1}
  \|J_1 \|_{L^{\Phi}(\Omega)} \lesssim \| g^* \|_{BMO(\Omega)}\| S(f) \|_1 \lesssim \| f \|_{H^1(\Omega)}\| g \|_{BMO(\Omega)}.
\end{equation}
As $S(\Pi_3(f,g)) \leq J_1 + J_2$, by combining 
\eqref{ms2} with \eqref{ms1}, and by the fact $L^1(\Omega)\subset L^\Phi(\Omega)$, we conclude
$$ \| \Pi_3(f,g)\|_{H^\Phi(\Omega)}=\| S(\Pi_3(f,g))\|_{L^\Phi(\Omega)} \lesssim \| f \|_{H^1(\Omega)}\| g \|_{BMO(\Omega)}, $$
as desired.

We would like to end this section with the comparison between our proof and the one provided in \cite{JXYZ}. Though both arguments heavily rely on the atomic decomposition of $H^1(\Omega)$, they further use weak atom decomposition for the diagonal Hardy space while our proof proceeds more directly. Moreover, the treatment of the most technical term $\Pi_3$ is significantly simplified in this section thanks to Theorem \ref{mBMO}.   

\section{Bilinear decompositions on $H^p(\Omega)\times \Lambda_1(\alpha_p)$ for $ 0<p<1 $ }\label{case2}
In this section, we give a proof of Theorem \ref{Theorem C}. Arguing as in the proof of  Theorem \ref{Theorem A}, it suffices to establish appropriate estimates for the bilinear operators $\Pi_1$, $\Pi_2$ and $\Pi_3$.

Let $(\Omega, \F, P)$ be a fixed probability space.
If we consider the filtration $ \F_0 = \{ \emptyset, \Omega \} $ and $ \F_k = \F$ for all $ k \geqslant 1 $,  then $H^p(\Omega) = L^p(\Omega)$ for $0<p<\infty$. It is well-known that $ (L^p(\Omega))^* \neq \{ 0 \} $ if and only if the probability space $ (\Omega, \F, P) $ contains at least one atom with non-zero measure when $0<p<1$. This means that $(H^p(\Omega))^*=\{0\}$ may occur. Therefore, we are only concerned with regular martingales where every $ \F_k$ is generated by countably many atoms.

To prove Theorem \ref{Theorem C}, we start with the following lemma, which holds for general martingales that are not necessarily regular. This shall be familiar to the experts in the area, but we will enclose the proof here for the sake of completeness. 
\begin{lemma}\label{lem4.2}
For any $0<p<1$, we have $L^1(\Omega)\subset H^p_m(\Omega)$.
\end{lemma}
\begin{proof}
 By Doob's maximal inequality, for any $f\in L^1(\Omega)$ and for any $ \lambda > 0 $ we have
\begin{equation}\label{JN}
  P(f^*>\lambda) \leqslant \dfrac{1}{\lambda}\int_{\lk{f^*>\lambda}}|f|dP.
\end{equation}
Without loss of generality, we may assume $ \| f \|_1 \leqslant 1 $. Then
\begin{align*}
  \| f^* \|_p^p & = \int_{\Omega}|f^*|^p dP = p \int_{0}^{\infty} P(|f^*|>\lambda)\lambda^{p-1} d\lambda \\
   & = p \int_{0}^{1} P(f^*>\lambda)\lambda^{p-1} d\lambda + p\int_{1}^{\infty}P(f^*>\lambda)\lambda^{p-1} d\lambda \\
   & \leqslant p \int_{0}^{1} \lambda^{p-1} d\lambda + p\int_{1}^{\infty}\dfrac{1}{\lambda}\sk{ \int_{\lk{f^*>\lambda}}|f|dP }\lambda^{p-1} d\lambda \\
   & = 1 + p\int_{\{f^*>1\}}|f|\sk{ \int_{1}^{f^*} \lambda^{p-2}d\lambda }dP \\
   & = 1 + \dfrac{p}{1-p}\int_{\{f^*>1\}}|f|\sk{ 1-|f^*|^{p-1} }dP \\
   & \leqslant 1 + \dfrac{p}{1-p}\int_{\{f^*>1\}}|f|dP \leqslant \dfrac{1}{1-p}.
\end{align*}
This implies that for any $f\in L^1(\Omega)$
$$
  \| f \|_{H^p_m(\Omega)} \leqslant \big(\dfrac{1}{1-p}\big)^\frac{1}{p} \| f \|_1,$$
which yields the desired result.
\end{proof}

For regular martingales, we have $L^1(\Omega)\subset H^p_m(\Omega)=H^p(\Omega)=h^p(\Omega)$. In what follows, the martingales are always assumed to be regular and every $ \F_k$ is generated by countable atoms.

\begin{corollary}\label{at of hp} For $0<p<1$ and $1\leq q\leq \8$, $H^p(\Omega) = H^{p, q}_{\rm at}(\Omega)$.
\end{corollary}
\begin{proof}
By considering the aforementioned atomic decomposition of $H^p(\Omega) $ and Definition \ref{ath}, we have $H^p(\Omega)=H^{p, \8}_{\rm at}(\Omega)$. It is easy to see $H^{p, \8}_{\rm at}(\Omega)\subset H^{p, q}_{\rm at}(\Omega)\subset H^{p, 1}_{\rm at}(\Omega)$. It thus suffices to show that $H^{p, 1}_{\rm at}(\Omega)\subset H^p(\Omega)$. By Lemma \ref{lem4.2}, if $a$ is a simple $(p,1)$-atom, then
$$ \|a\|_{H^p(\Omega)}\lesssim_p \|a\|_1 , $$
which implies that $a\in H^p(\Omega)$. Hence, $H^{p,1}_{\rm at}(\Omega)\subset H^{p}(\Omega)$ and so, $H^p(\Omega)=H^{p, q}_{\rm at}(\Omega)$.
\end{proof}

\subsection{Characterization of martingale Lipschitz spaces}
In this subsection, we give a characterization of martingale Lipschitz spaces that appears to be new and useful in our argument.
We shall first recall the definition of martingale Lipschitz spaces. For $0 < p < 1$ define
\begin{equation}\label{Lip}
  \Lambda_q(\alpha_p):=\lk{ f\in L^2(\Omega): \|f\|_{\Lambda_q(\alpha_p)}=\sup_{n\geq 0}\sup_{A\in\F_n} P(A)^{-\frac1q-\alpha_p} \sk{ \int_A |f-f_n|^q dP }^{\frac1q} <\infty },
\end{equation}	
where $q=1$ or $q=2$, $\alpha_p:= \frac{1}{p}-1 >0$.

In \cite{WF1}, Weisz showed that $ (H^p(\Omega))^* = \Lambda_1(\alpha_p) $ and $ \Lambda_1(\alpha_p)=\Lambda_2(\alpha_p) $.

\begin{corollary}\label{lam1}
For any $ g\in \Lambda_1(\alpha_p)$, we have $ \| g-g_0 \|_{\infty} \lesssim_p \|  g \|_{\Lambda_1(\alpha_p)} $.
\end{corollary}

\begin{proof}
By duality and Lemma \ref{lem4.2}, for any $ f \in L^2(\Omega)$,
$$ | \ce\big(f(g-g_0)\big)|=\big| \ce\sk{g(f-f_0)} \big|\lesssim_p \| f \|_{H^p} \| g \|_{\Lambda_1(\alpha_p)} \lesssim_p \| f \|_{1} \| g \|_{\Lambda_1(\alpha_p)}. $$
The above estimate together with the fact $\big(L^1\big(\Omega))^*=L^\8(\Omega)$ yields
 $$ \| g-g_0 \|_{\infty} \lesssim_p \| g \|_{\Lambda_1(\alpha_p)}, $$
  which finishes the proof.
\end{proof}
 By virtue of Corollary \ref{lam1}, we have the following  property of martingale Lipschitz spaces.
\begin{thm}\label{thm4.5}
If $ g\in \Lambda_1(\alpha_p) $, we have $ \|1_A \cdot | g-g_n |\|_\8 \lesssim_p P(A)^{\alpha_p} \| g \|_{\Lambda_1(\alpha_p)} $,  for any $ n\in \nat$ and any $ A \in \F_n$.
\end{thm}
\begin{proof}
  Note that when $P(A)=0$, the desired result holds trivially. Fix $n\in \nat$ and $ A \in \F_n$ with $ P(A)\neq 0 $. For $k\geq 0$, let  $ \F_k^A := \lk{ B \in \F_{k+n}: B \subseteq A }$. Note that the union $\F^A$ of all $ \F_k^A$ is exactly  $\{B\in \F | B\subset A\}$. Hence, if we define 
$$ P_A(B) := \dfrac{P(B)}{P(A)} \quad (B\in \F^A)  $$
then $(A, \F^A, P_A)$ is a probability space. Note that for any $g\in L^1(A, \F^A, P_A)$ one has
$$ \mathbb{E}(g| \F^A_k)=1_A \cdot \ce(g| \F_{k+n}) . $$
Denote $ \mathbb{E}( \cdot | \F^A_k)$ by $\ce_k^A$. It is easy to verify $\{ \ce_k^A(g) \}_{k\geq0}$ is also a regular martingale on $(A, \F^A, P_A)$.  If $g\in  \Lambda_1(\alpha_p)$, then for $B\in \F_k^A$ with $P(B)\neq 0$,
\begin{equation*}
 \begin{aligned}
 P_A(B)^{-1-\alpha_p}\sk{\int_B |g-\ce_k^A (g)| dP_A}&=P(A)^{\alpha_p} \sk{P(B)^{-1-\alpha_p}\sk{\int_B |g-g_{k+n}| dP}}\\
&\leq P(A)^{\alpha_p} \|g\|_{\Lambda_1(\alpha_p)}
\end{aligned}
\end{equation*}
which implies that by Corollary \ref{lam1},
\begin{equation*}
 \| 1_A \cdot |g-g_n|\|_\8=\| 1_A \cdot |g-\ce_0^A (g)|\|_\8 \lesssim_p P(A)^{\alpha_p}\| g \|_{\Lambda_1(\alpha_p)}.
\end{equation*}
This completes the proof of the theorem.
\end{proof}

\begin{remark}\label{rmk_Lip}
By Theorem \ref{thm4.5} and \eqref{Lip}, we conclude that for $g\in \Lambda_1(\alpha_p)$ we have the characterization
\begin{equation}\label{rmk_Lip_eq}
 \|g\|_{\Lambda_1(\alpha_p)}\approx_p \sup_{n\geq 0}\sup_{A\in\F_n} P(A)^{-\alpha_p}  \| 1_A \cdot |g-g_n|\|_\8. 
\end{equation}
Note that the results in \cite{MN} can be deduced from \eqref{rmk_Lip_eq}.
\end{remark}

\subsection{Proof of Theorem \ref{Theorem C}}
As in the proof of Theorem \ref{Theorem A}, we divide the proof into three parts. Without loss of generality, we may assume that $f_0=g_0=0$.

\subsubsection{Estimates for $ \Pi_1(f,g) $ and $ \Pi_3(f,g) $} The boundedness of $\Pi_1$ from $H^p(\Omega)\times \Lambda_1(\alpha_p)$ to $L^1(\Omega)$ follows directly from the duality between $H^p(\Omega)$ and $\Lambda_1(\alpha_p)$,   
 we omit the details.

We shall also prove that $ \Pi_3 $ is a bounded bilinear operator from $ H^p(\Omega) \times \Lambda_1(\alpha_p) $ to $ H^p(\Omega) $. Note that 
\begin{equation}\label{pi-3-3}
  S(\Pi_3(f,g))^2 = \sum_{k=1}^{\infty} |g_{k-1}|^2 |d_kf|^2 \leqslant (g^*) ^2 S(f)^2 .
\end{equation}
Hence we conclude from Corollary \ref{lam1} and the $L^\infty$ boundedness of the maximal function that 
\begin{equation}\label{pi-3-2}
  \| \Pi_3(f,g) \|^p_{H^p(\Omega)} \lesssim \|g^*\|_\infty^p  \ce (S(f)^p) \leq \|g\|_\infty^p  \| f \|_{H^p(\Omega)}^p   \lesssim_p \| f \|_{H^p(\Omega)}^p \| g \|_{\Lambda_1(\alpha_p)}^p,
\end{equation}
as desired. 

\subsubsection{Estimates for $ \Pi_2(f,g) $.}  
We will show that $ \Pi_2 $ is a bounded bilinear operator from $ H^p(\Omega) \times \Lambda_1(\alpha_p) $ to $ H^1(\Omega) $.
Note that $H^p(\Omega)=h^p(\Omega)$, and $h^p(\Omega)$ admits an atomic decomposition.

Then there exist $ \lk{\lambda_j}_{j\geqslant 1} \subset \real$ and simple $(p,\8)$-atoms $ \lk{a^j}_{j \geqslant 1} $ such that
\begin{equation}\label{atom dec}
f = \sum_{j=1}^{\infty} \lambda_j a^j,\quad \| f \|_{H^p(\Omega)} \approx_p \sk{\sum_{j=1}^{\infty} |\lambda_j|^p}^\frac{1}{p},
\end{equation}
where we assume $\supp (a^j) \subset A_{n_j} $ and $ A_{n_j} \in \F_{n_j} $ with $P(A_{n_j})>0$ for $j\geq 1$.
By arguing as in the corresponding case in the proof of Theorem \ref{Theorem A},
\begin{equation}
  S(\Pi_2(a^j,g)) = \sk{\sum_{k=n_j+1}^\8 1_{A_{n_j}}|a^j_{k-1}|^2|d_kg|^2}^{\frac{1}{2}} \leqslant |(a^j)^*|\sk{\sum_{k=n_j+1}^\8 1_{A_{n_j}}|d_kg|^2}^{\frac{1}{2}}.
\end{equation}
Hence,
\begin{align*}
\ce\mk{S(\Pi_2(a^j,g))} & \leqslant \| (a^j)^* \|_{\infty}\mk{ \ce\sk{1_{A_{n_j}} \sum_{k=n_j+1}^\8|d_kg|^2 }^{\frac{1}{2}} } \\
& \leqslant \| a^j \|_{\infty}\sk{P(A_{n_j})\sk{ \ce \sum_{k=n_j+1}^\8|d_kg|^2 }}^{\frac{1}{2}} \\
& \leqslant P(A_{n_j})^{-\frac{1}{p}}\sk{P(A_{n_j})\| g \|_{\Lambda_2(\alpha_p)}^2 P(A_{n_j})^{1+2\alpha_p} }^{\frac{1}{2}} \\
& = \| g \|_{\Lambda_2(\alpha_p)} P(A_{n_j})^{-\frac{1}{p}}P(A_{n_j})^{1+\alpha_p} \\
& \leqslant \| g \|_{\Lambda_2(\alpha_p)} \lesssim_p \| g \|_{\Lambda_1(\alpha_p)},
\end{align*}
where the last inequality follows from the condition that $\alpha_p = \frac{1}{p}-1$. As a consequence of the above estimates, we have that
\begin{equation}\label{pip-2}
\| \Pi_2(f,g) \|_{H^1(\Omega)}^p \leqslant \sum_{j=1}^\8|\lambda_j|^p\mk{\ce S(\Pi_2(a^j,g))}^p \lesssim_p \| f \|_{H^p(\Omega)}^p\| g\|_{\Lambda_1(\alpha_p)}^p.
\end{equation}
This completes the proof of the theorem.

\section{Homogeneous spaces}\label{case3}
In this section, we introduce some fundamental concepts and important theorems for homogeneous spaces, which can be found in \cite{Co}. We begin with the definition of homogeneous spaces. Recall that $ d $ is a quasi-metric on $\Omega$ if
\begin{enumerate}
\item $d(x,x)\geq 0$, \ $\forall x \in \Omega$;
\item $d(x,y)=d(y,x)$, \ $\forall x,y \in \Omega$;
\item there exists a constant $ A_0 \geq 1$ such that
\begin{equation}\label{metric}
d(x,y) \leqslant A_0(d(x,z)+ d(z,y)) , \quad \forall x,y,z \in \Omega.
\end{equation}
\end{enumerate}

Denote by $B(x,r):= \{y \in \Omega : d(y, x) < r\}$ the open ball centered at $x$ with radius $r$. In this paper, all quasi-metric spaces are assumed to have the doubling property: there exists a positive integer $A_1 \in \nat$ such that for every $ x \in \Omega $ and for every
$r > 0$, the ball $B(x, r) $ can be covered by at most $A_1$ balls $B(x_i
, \frac{r}{2})$ for some $x_i\in \Omega$.

\begin{definition}
A $ \sigma$-finite measure space $ ( \Omega, \F, \mu ) $ equipped with a quasi-metric $d$ is called a homogeneous space if $ \mu $ is a Borel measure of homogeneous type:
\begin{equation}\label{double}
0<\mu \sk{B(x,2r)} \leqslant 2^{C_\mu} \mu \sk{B(x,r)}<\8 ,\quad \forall x \in \Omega, r>0, 
\end{equation}
where the constant $ C_{\mu} $ is independent of $ x $ and $ r $.
\end{definition}

In \cite{Co}, Coifman and Weiss defined Hardy spaces on homogeneous spaces by regarding their elements as linear functionals acting on some appropriate quasi-normed spaces. In order to state the definition of Coifman and Weiss, we need to introduce the notions of atoms, $BMO$ and Lipschitz spaces on homogeneous spaces.

\begin{definition} 
If $ 0 <  p \leqslant 1 \leqslant q \leq \infty $ and $p<q$, we say that a function $ a $ is a $(p, q)$-atom if
\begin{enumerate}
\item $ \supp (a) \subset B $ where $ B $ is a ball;
\item $\|a\|_q \leqslant \sk{ \mu(B) }^{\frac{1}{q}-\frac1p}$;
\item $ \int_\Omega a d\mu =\int_B a d\mu = 0. $
\end{enumerate}
\end{definition}

\begin{definition}
A locally integrable function $f$ is called a $BMO$ function if
$$ \| f \|_{BMO} := \sup_{B} \dfrac{1}{\mu(B)} \int_{B} |f-f_B|d\mu < \infty, $$
where $f_B:=\dfrac{1}{\mu(B)} \int_{B} fd\mu$, and the supremum runs over all balls $B$. Denote by $ BMO(\mu) $ the  $BMO$ space consisting of all $BMO$ functions.
\end{definition}

\begin{definition} 
For $ \alpha > 0 $, a locally integrable function $ l  $ is called a Lipschitz function if
\begin{equation}\label{lip}
| l(x) - l(y) | \leqslant C_{\alpha}\sk{ \mu(B) }^{\alpha} \text{ for any $ x,y \in \Omega $ and any ball $  B $ containing $ x,y $.  }
\end{equation}
Moreover,
\begin{equation}
\| l \|_{\L_\alpha} := \inf\{ C_\alpha:| l(x) - l(y) | \leqslant C_{\alpha}\sk{ \mu(B) }^{\alpha}, \ \forall x,y \in B \},
\end{equation}
where the infimum runs over all balls $ B $. Denote by $ \L_\alpha(\mu) $ the space consisting of all Lipschitz functions.
\end{definition}

It is well-known that each $ BMO $ function can be regarded as a continuous linear functional on the vector space generated by finite linear combinations of $(1, q)$-atoms for $1<q\leq \8$ (cf. \cite{Co}). Hence we can define the atomic Hardy space $ H^{1, q}_{\rm at}(\mu) $ $(1<q\leq \8)$ as follows:
\begin{multline}\label{De1}
H^{1, q}_{\rm at}(\mu) := \\
 \left\{ f\in \sk{BMO(\mu)}^*  : \  
  f = \sum_{j=0}^{\infty} \lambda_j a^j, \text{ where $a^j$ is a $(1,q)$-atom and } \sum_{j=0}^\8|\lambda_j| < \infty \right\}
\end{multline}
endowed with the norm
$$ \|f\|_{H^{1, q}_{\rm at}(\mu)}:= \inf \lk{\sum_{j=0}^\8|\lambda_j|: f = \sum_{j=0}^{\infty} \lambda_j a^j, \text{ where $a^j$ is a $(1,q)$-atom} } . $$

Similarly, each Lipschitz function $ l \in \L_{\alpha_p}(\mu) $ can be also regarded as a continuous linear functional of the vector space generated by finite linear combinations of $ (p,q) $-atoms where $ 0<p<1\leq q\leq \8$ and $ \alpha_p = \frac{1}{p} - 1 $ (cf. \cite{Co}). We define the atomic Hardy spaces $ H^{p, q}_{\rm at}(\mu) $ as follows:
\begin{multline}\label{De2}
H^{p, q}_{\rm at}(\mu) \\ := \left\{ f\in \sk{\L_{\alpha_p}(\mu)}^* : \  f = \sum_{j=0}^{\infty} \lambda_j a^j, \text{ where $a^j$ is a $(p,q)$-atom and } \sum_{j=0}^\8|\lambda_j|^p < \infty \right\}
\end{multline}
endowed with the quasi-norm
$$ \|f\|_{H^{p, q}_{\rm at}(\mu)}:= \inf \lk{\sk{\sum_{j=0}^\8|\lambda_j|^p}^\frac{1}{p}: f = \sum_{j=0}^{\infty} \lambda_j a^j, \text{ where $a^j$ is a $(p,q)$-atom} } . $$

Although the Hardy spaces vary with $p$ and $q$ according to the above definitions, the following theorem, which can be found in \cite{Co}, shows that the Hardy spaces actually depend only on $p$. Consequently, this enables us to define the Hardy spaces $ H^p_{\rm at}(\mu) $ for $ 0<p\leqslant 1 $ to be any one of the  spaces $ H^{p, q}_{\rm at}(\mu) $ for $0< p<q \leqslant \infty, 1\leqslant q \leq \8 $.

\begin{thm}\label{thm2.15} $ H^{p, q}_{\rm at}(\mu) = H_{\rm at}^{p,\infty}(\mu) $ whenever $ 0 < p \leq 1 \leqslant q \leqslant \infty $ and $p<q$.
\end{thm}

We end this section with the following duality theorem obtained in \cite{Co}.
 
\begin{thm}\label{dualhp}
$ \sk{H^1_{\rm at}(\mu)}^*= BMO(\mu) $, and 	$ \sk{H^p_{\rm at}(\mu)}^*= \L_{\alpha_p}(\mu) $ for $ 0<p<1 $.
\end{thm}

The proofs of Theorem \ref{thm2.15} and Theorem \ref{dualhp} that appeared in \cite{Co} are very technical. In the following sections, by employing martingale methods, we give much simpler proofs of these facts. Our approach is based on the fact that $H^p_{\rm at}(\mu)$ for $0<p\leq 1$ is the finite sum of several dyadic martingale Hardy spaces.

\section{Dyadic systems on homogeneous spaces}\label{section6}
In this section, we start with introducing dyadic systems on homogeneous spaces, which first appeared in the work of Hyt\"{o}nen and Kairema \cite{Hy}. With the help of these dyadic structures, we then show that $H^p_{\rm at}(\mu)$ is exactly the finite sum of martingale Hardy spaces associated with some adjacent dyadic martingales, which extends Mei's result \cite{MT} to homogeneous spaces.

The following theorem concerning the existence of dyadic structures is due to  Hyt\"{o}nen and Kairema \cite{Hy}.
\begin{thm}\label{dyasys} Let $\Omega$ denote a homogeneous space. Suppose that the constants $0 < c_0 \leqslant C_0 < \infty $ and $\delta \in (0, 1) $ satisfy
$$ 12 A_0^3C_0\delta \leqslant c_0, $$
where $A_0$ is specified in the definition of quasi-metric, see \eqref{metric}. 

Given a set of reference points $\{z^k_\alpha\}_\alpha,\ \alpha \in \A_k $ (an index set), for every $k \in \ent$, with the properties that
$$ d(z^k_\alpha,z^k_\beta) \geqslant c_0\delta^k,\ (\alpha \neq \beta) \quad \min_\alpha d(x,z^k_\alpha) < C_0 \delta^k, \text{ for all } x\in \Omega, $$
one can construct families of sets $\tilde{Q}^k_\alpha \subseteq Q^k_\alpha \subseteq \bar{Q}^k_\alpha$, called open, half-open and closed dyadic cubes respectively, such that:
\begin{eqnarray}
&\tilde{Q}^k_\alpha \text{ and } \bar{Q}^k_\alpha \text{are the interior and closure of } Q^k_\alpha; \label{dya1}\\
&\text{if } k \leqslant l ,\text{ then either } Q^l_\beta \subseteq Q^k_\alpha \text{ or } Q^l_\beta \cap Q^k_\alpha = \emptyset; \\
&X = \bigcup\limits_\alpha Q^k_\alpha \text{ (disjoint union)  for all } k \in \ent; \\
&B(z^k_\alpha,c_1\delta^k) \subseteq Q^k_\alpha \subseteq B(z^k_\alpha,C_1\delta^k) =: B(Q^k_\alpha) \text{ where } c_1 = (3A_0^2)^{-1}c_0 \text{ and } C_1 = 2A_0C_0; \\
&\text{if } k \leqslant l \text{ and } Q^l_\beta \subseteq Q^k_\alpha \text{ then } B(Q^l_\beta) \subseteq B(Q^k_\alpha)\label{dya5}.
\end{eqnarray}
The open and closed cubes $ \tilde{Q}^k_\alpha $ and $ \bar{Q}^k_\alpha $ depend only on the points $ z^l_\beta $ for $ l \geqslant k $. The half-open cubes $Q^k_\alpha$ depend on $ z^l_\beta $ for $ l \geqslant \min(k, k_0) $, where $ k_0 \in \ent $ is a preassigned number entering the construction.
\end{thm}

It is obvious that the construction of the above dyadic systems is not unique, and it depends on the set of the reference points $\{z^k_\alpha\}_\alpha$. We denote this dyadic system by $ \D = \{ Q^k_\alpha \}_{k,\alpha} $. Let $ \F_k = \sigma( \{ Q^k_\alpha \}_{\alpha} ) $ be the $ \sigma $-algebra generated by $\{ Q^k_\alpha \}_{\alpha}$. Then it is clear that
$$ \cdots \subset \F_{k-1} \subset \F_k \subset \cdots , $$
which implies that $ \{ \F_k \}_{k \in \ent} $ is a filtration generated by atoms. Let $\F=\sigma{(\cup_{k\in \ent}\F_k)}$. Note that each $Q^k_\alpha \text{ is an atom of } \F_k$.

\begin{rem}  The standard dyadic grid on the real line is a dyadic system given by
$$  \F_k=\{[2^{-k}m,2^{-k}(m+1)) : m \in \ent\} \quad \text{ for all }  k \in \ent. $$
Similarly, an example of a dyadic system on  $ \real^n$ is given by the family of standard dyadic cubes in $ \real^n$.
\end{rem}

Recall that, for $f\in L^1_{\text{loc}}(\Omega, \F, \mu)$, the martingale maximal function, the square function and the conditional square function of $f$ associated with $ (\F_k)_{k \in \ent} $ are given by
$$ f^{*}: = \max\limits_{k \in \ent}|f_k|, \quad S(f):=\sk{ \sum\limits_{k \in \ent} |d_kf|^2}^\frac{1}{2} \quad  \text{and} \quad s(f):=\sk{ \sum\limits_{k \in \ent} \ce_{k-1}|d_kf|^2}^\frac{1}{2} , $$
respectively.

Let $0<p\leq 1$. The martingale Hardy space $ H^p_{m,\D}(\mu)$ is defined as the completion of the space consisting of all $f \in  L^1_{\text{loc}}  (\Omega) $ such that $ f^{\ast} \in L^p  (\Omega)$ with respect to the quasi-norm $\|f\|_{H^p_{m,\D}(\mu)}:=\|f^*\|_p$. 

We define $H^p_{\D}(\mu) $ and $h^p_{\D}(\mu)$ by the square functions and the conditional square functions respectively, with the additional assumption that
\begin{equation}\label{conmat}
\lim\limits_{n\rightarrow-\8} \int_{\Omega} \sup_{k\leq n}|f_k|^p d\mu=0 . 
\end{equation}
 From \eqref{conmat}, we have
 $$ \lim\limits_{n\rightarrow -\8} \sup_{ k\leq n} |f_k|=0. $$

Analogously, define the martingale atomic Hardy spaces $ H^{p, q}_{\rm at,\D}(\mu) $ $(0<p<1\leq q\leq\8 \ \text{or}\ p=1, 1<q\leq\8)$ like Definition \ref{ath}. 
 
In order to show Theorem \ref{dualhp}, we introduce the dual spaces of these atomic martingale Hardy spaces. For $0<p<1$, $q=1 \ \text{or} \ 2$ and $\alpha_p=\frac{1}{p}-1$, define

\begin{align*}
BMO^{\D}(\mu) &:= \lk{f\in L^1_{\text{loc}}(\Omega, \mu)\ : \ \| f \|_{BMO^{\D}(\mu)} := \sup_{Q \in \D} \dfrac{1}{\mu(Q)} \int_{Q} |f-f_Q|d\mu < \infty},\\
\Lambda_q^{\D}(\alpha_p) &:=\lk{f\in L^1_{\text{loc}}(\Omega, \mu)\ : \ \| f \|_{\Lambda_q^{\D}(\alpha_p)} := \sup_{Q \in \D} {\mu(Q)}^{-\frac{1}{q}-\alpha_p} \sk{\int_{Q} |f-f_Q|^q d\mu }^\frac{1}{q}< \infty}.
\end{align*}
The spaces $\Lambda_q^{\D}(\alpha_p)$ are called the martingale Lipschitz spaces with respect to $\D$. Note that $\Lambda_1^{\D}(\alpha_p) =\Lambda_2^{\D}(\alpha_p) $.
 
Arguing as in \cite{WF1}, one can show that
$$  \sk{H^{1}_{\rm at,\D}(\mu)}^*=BMO^{\D}(\mu),$$
and for $0<p<1$,
$$
  \sk{H^{p}_{\rm at,\D}(\mu)}^*=\Lambda_q^{\D}(\alpha_p).
$$
\begin{rem}
Since every simple $(p,q)$-atom is locally supported, by Corollary \ref{at of hp}, we conclude that for $ 0<p<1\leq q\leq\8 \ \text{or}\ p=1, 1<q\leq\8 $
$$ H^{p, q}_{\rm at,\D}(\mu)=H^{p, \8}_{\rm at,\D}(\mu).  $$
Thus we are only concerned with $H^{p}_{\rm at,\D}(\mu):=H^{p, \8}_{\rm at,\D}(\mu)$.
\end{rem}

\begin{proposition}\label{decHm}
For $ 0<p\leqslant 1 $, the martingale Hardy spaces defined above are mutually equivalent. Namely, $ H^p_{\D}(\mu)=H^p_{m,\D}(\mu)=h^p_{\D}(\mu)=H^{p}_{\rm at,\D}(\mu)$.
\end{proposition}

\begin{proof} Let $p \in (0,1]$ be fixed. First, we show $H^p_{\D}(\mu)=H^p_{m,\D}(\mu)$. Suppose that $f\in H^p_{m,\D}(\mu)$. Then for any $n>0$, by a well-known inequality of Burkholder--Davis--Gundy,
$$ \int_{\Omega} \sk{|f_{-n}|^2+\sum_{k=-n+1}^{n}|d_kf|^2}^\frac{p}{2} d\mu\lesssim \int_{\Omega}\sup_{-n\leq k\leq n}|f_k|^pd\mu \lesssim   \int_{\Omega}(f^*)^p d\mu $$
which yields by letting $n\rightarrow\8$ and by Fatou's lemma
$$ \|S(f)\|_p\lesssim \|f^*\|_p.  $$
Thus $H^p_{m,\D}(\mu)\subset H^p_{\D}(\mu)$.
	
Conversely, if $f\in H^p_{\D}(\mu)$, then for $n>0$,
$$ \int_{\Omega}\sup_{-n\leq k\leq n}|f_k|^pd\mu  \lesssim  \int_{\Omega} \sk{|f_{-n}|^2+\sum_{k=-n+1}^{n}|d_kf|^2}^\frac{p}{2} d\mu, $$
and hence
 \begin{equation}\label{6.7}
 \int_{\Omega}\sup_{-n\leq k\leq n}|f_k|^pd\mu \lesssim \int_{\Omega}\sup_{k\leq -n}|f_{k}|^p d\mu + \int_{\Omega} |S(f)|^p d\mu<\8  .
 \end{equation}
 Then by letting $n\rightarrow\8$ and applying Fatou's lemma, we obtain $ \|f^*\|_p <\8 $ and 
 $$  \|f^*\|_p\lesssim \|S(f)\|_p. $$
 Therefore, $H^p_{\D}(\mu)\subset H^p_{m,\D}(\mu)$ and   $H^p_{m,\D}(\mu)= H^p_{\D}(\mu)$. 
    
One shows $H^p_{m,\D}(\mu)= h^p_{\D}(\mu)$ in a completely analogous way. To show $h^p_{\D}(\mu)=H^{p}_{\rm at,\D}(\mu)$, one can argue by mimicking the corresponding proof in \cite{WF1} and \cite{WF2}. We omit the details.
\end{proof}

The following theorem can be found in \cite{Hy} and ensures that there exist enough dyadic cubes to cover all balls on homogeneous spaces.

\begin{thm}\label{kdya}
Given a set of reference points $\{z^k_\alpha\}, k \in \ent, \alpha \in \A_k$, suppose that there exists constant $\delta \in (0, 1) $ that satisfies $96A^6_0\delta \leqslant 1$. Then there exists a finite collection of families $\D^t, t = 1, 2, \cdots, K = K(A_0, A_1, \delta) < \infty $, where each $\D^t$ is a collection of dyadic cubes, associated to dyadic points $\{z^k_\alpha\}, k \in \ent, \alpha \in \A_k $, with the properties \eqref{dya1}-\eqref{dya5} in Theorem \ref{dyasys}. 
	
In addition, the following property is satisfied:
\begin{equation}\label{ballcube}
\text{for every }  B(x, r) \subseteq \Omega, \text{ there exist $ t $ and $Q \in \D^t$  with $B(x,r) \subseteq Q$ and $diam(Q) \leqslant Cr$.}
\end{equation}
The constant $C < \infty$ in \eqref{ballcube} only depends on the quasi-metric constant $A_0$ and the parameter $\delta$.
\end{thm}

By virtue of Proposition \ref{decHm} and Theorem \ref{kdya}, we have the following theorem, which extends Mei's  result in \cite{MT}.
\begin{thm}\label{Hardy} For $ 0<p\leqslant 1$, we have
\begin{equation}\label{H to Hd}
H^p_{\rm at}(\mu) = \sum_{t=1}^{K} H^p_{\rm at,\D^t}(\mu) = \sum_{t=1}^{K} H^p_{\D^t}(\mu) = \sum_{t=1}^{K} H^p_{m,\D^t}(\mu)=\sum_{t=1}^{K} h^p_{\D^t}(\mu) .
\end{equation}
\end{thm}

\begin{proof} Let $p\in (0,1]$ be fixed. In view of  Proposition \ref{decHm}, it suffices to show $ H^p_{\rm at}(\mu) = \sum\limits_{t=1}^{K} H^p_{\rm at,\D^t}(\mu) $. We prove it via comparing the atoms. Let $ a$ be a $ (p,\8) $-atom in $H^p_{\rm at}(\mu) $. Then there exists a ball $ B $ such that
$$ \supp(a) \subset B,\ \| a \|_\8 \leqslant \sk{\mu(B)}^{-\frac1p},\ \int_B a(x) d\mu = 0. $$
	
By Theorem \ref{kdya}, there exist $ t $ and a cube $ Q \in \D^t $ such that $ B \subset Q $, and $ \mu(Q) \lesssim \mu(B)$. 
Then
$$ \supp(a) \subset B \subset Q,\ \| a \|_\8 \leqslant \sk{\mu(B)}^{-\frac1p} \lesssim \sk{\mu(Q)}^{-\frac1p} ,\ \int_Q a d\mu = 0, $$
which implies that $a$ is a constant multiple of a simple $(p,\8)$-atom in $ H^p_{\rm at,\D^t}(\mu) $. Thus
\begin{equation}\label{H to Hd1}
H^p_{\rm at}(\mu) \subset \sum_{t=1}^{K} H^p_{\rm at,\D^t}(\mu).
\end{equation}
	
For any $t=1,2,\cdots,K$ and for any given simple $ (p,\8) $-atom $ b$ in $H^p_{\rm at,\D^t}(\mu) $, there exists $ Q \in \D^t $ such that
$$ \supp(b) \subset Q,\ \| b \|_\8 \leqslant \sk{\mu(Q)}^{-\frac1p},\ \int_Q b d\mu =  0. $$
By Theorem \ref{dyasys}, there exists a ball $ B $ such that $ Q \subset B $ and $\mu(Q)\gtrsim \mu(B)$. Hence
 $$ \supp(b) \subset Q \subset B,\ \| b \|_\8 \leqslant \sk{\mu(Q)}^{-\frac1p} \lesssim \sk{\mu(B)}^{-\frac1p},\ \int_B b  d\mu =  0, $$
which implies that a multiple of $ b $ is also a $(p,\8)$-atom in $ H^p_{\rm at}(\mu) $, thus
\begin{equation}\label{H to Hd2}
\sum_{t=1}^{K} H^p_{\rm at,\D^t}(\mu) \subset H^p_{\rm at}(\mu).
\end{equation}
To complete the proof of the theorem, combine \eqref{H to Hd1} and \eqref{H to Hd2}.
\end{proof}

\begin{rem} Theorem \ref{thm2.15} follows immediately from Corollary \ref{at of hp}, Proposition \ref{decHm} and Theorem \ref{Hardy}, which simplifies the original proof by Coifman and Weiss in \cite{Co}.
\end{rem}

By duality and Theorem \ref{Hardy}, we recover  the following result of  \cite{Hy}, which is an extension of a result due to Mei \cite{MT}: 
\begin{equation}\label{bmo}
BMO(\mu) = \bigcap_{t=1}^{K} BMO^{\D^t}(\mu).
\end{equation}
We will now establish an analogous result for $\L_{\alpha_p}(\mu)$ $(0<p<1)$.

\begin{thm}\label{Lips}
For $0<p<1$,
$$ \L_{\alpha_p}(\mu) = \bigcap_{t=1}^{K} \Lambda_2^{\D^t}(\alpha_p). $$
\end{thm}

\begin{proof} By Theorem \ref{dyasys}, for any $ Q \in \D^t $ (and $ t = 1,2,\cdots,K $), there exists a ball $B$ such that $ Q\subset B$ and $\mu(B)\lesssim \mu(Q)$.
If $ f \in \L_{\alpha_p}(\mu) $, then for any $x, y\in Q$, we have
$$ |f(x) - f(y)| \leq \| f \|_{\L_{\alpha_p}(\mu)} \mu(B)^{\alpha_p}\lesssim \| f \|_{\L_{\alpha_p}(\mu)} \mu(Q)^{\alpha_p}. $$
We thus have
\begin{align*}\label{lamd2}
\| f \|_{\Lambda_2^{\D^t}(\alpha_p)} &\leq \sup_{Q\in \D^t} \sk{ \mu(Q) }^{-\frac12-\alpha_p}\sk{ \mu(Q)^{-2}\int_Q \left( \int_Q |f(x) - f(y)| d\mu(y) \right)^2 d\mu(x) }^\frac12  \\
 & \leqslant\sup_{Q\in \D^t} \sk{ \mu(Q) }^{-\frac12-\alpha_p}\sk{ \int_Q \| f \|^2_{\L_{\alpha_p}(\mu)} \mu\sk{ Q }^{2\alpha_p} d\mu }^\frac12  \\
& \lesssim \| f \|_{\L_{\alpha_p}(\mu)},
\end{align*}
which yields
\begin{equation}\label{LIP1}
\L_{\alpha_p}(\mu) \subset \bigcap_{t=1}^{K} \Lambda_2^{\D^t}(\alpha_p) .
\end{equation}

Conversely, let $f\in  \bigcap\limits_{t=1}^{K} \Lambda_2^{\D^t}(\alpha_p) .$ For $Q\in \D^t$, by Theorem \ref{thm4.5}, 
$$ | f(x) - f_Q | \lesssim \mu(Q)^{\alpha_p}\| f \|_{\Lambda_2^{\D^t}(\alpha_p)} \ \ \ \forall x\in Q, $$
which implies that for any $x, y\in Q$,
\begin{equation}\label{hcm}
| f(x) - f(y) | \lesssim \mu(Q)^{\alpha_p}\| f \|_{\Lambda_2^{\D^t}(\alpha_p)}.
\end{equation}
	
For any ball $B\subset \Omega$, by Theorem \ref{kdya}, there exist $ t $ and $ Q \in \D^t $ such that $ B \subset Q$ and $  \mu(Q) \lesssim \mu(B).$ Then for any $x, y\in B$, by \eqref{hcm}
$$ | f(x) - f(y) | \lesssim\mu(B)^{\alpha_p}\| f \|_{\Lambda_2^{\D^t}(\alpha_p)}.   $$
Thus
\begin{equation*}
\| f \|_{\L_{\alpha_p}} \lesssim \sum_{t=1}^{K} \| f \|_{\Lambda_2^{\D^t}(\alpha_p)},
\end{equation*}
which implies
\begin{equation}\label{LIP2}
\bigcap_{t=1}^{K} \Lambda_2^{\D^t}(\alpha_p) \subset \L_{\alpha_p}(\mu).
\end{equation}
	
The theorem follows from \eqref{LIP1} and \eqref{LIP2}.
\end{proof}

\begin{remark}
Theorem \ref{Hardy} and Theorem \ref{Lips} give a simple proof of Theorem \ref{dualhp} originally established by Coifman and Weiss \cite{Co}:
$$ (H^p_{\rm at}(\mu))^* = \sk{ \sum_{t=1}^{K} H^p_{\rm at,\D^t}(\mu)}^* = \bigcap_{t=1}^{K}(H^p_{\rm at,\D^t}(\mu))^*  = \bigcap_{t=1}^{K} \Lambda_2^{\D^t}(\alpha_p) = \L_{\alpha_p}(\mu). $$
\end{remark}

\section{Bilinear decompositions for dyadic martingales on homogeneous spaces}\label{section7}

In this section, we focus on bilinear decompositions arising in the study of products between elements in spaces of dyadic martingales on homogeneous spaces introduced in the previous section. In the setting of homogeneous spaces, due to their quasi-metrics and measures, the dyadic martingales behave worse than martingales in probability spaces and the underlying analysis is more intricate.

In \S  \ref{gHI_hs} we prove appropriate generalized H\"{o}lder-type inequalities on homogeneous spaces (see Lemmas \ref{Holder1} and \ref{Holder2} below). We then introduce a class of pointwise multipliers of $ \Lambda_{1,+}^\D(\alpha_p)$ and $BMO^\D(\mu)$; see Theorem \ref{thm7.4} below. Using Theorem \ref{thm7.4}, we define products between dyadic martingale Hardy spaces on homogeneous spaces and their duals and then, in \S \ref{bd_hs} we establish analogues of the results of Sections \ref{case1} and \ref{case2} in the setting of homogeneous spaces.

\subsection{A generalized H\"{o}lder-type inequality}\label{gHI_hs} Let $0<p\leq 1$ and $\D$ be a dyadic system, constructed as in Theorem \ref{dyasys}. The martingale Musielak--Orlicz Hardy spaces $H^{\Psi_p}_\D(\mu)$ consist of all measurable functions $f$ on $(\Omega, \F, \mu)$ such that $s(f)\in L^{\Psi_p}(\Omega)$ where $ O \in \Omega $ is a fixed point, and
\begin{align*}
\Psi_1 (x,t)& := \dfrac{t}{\log\sk{e+d(x,O)} + \log(e+t) }, \\
 \Psi_p(x,t)& := \frac{t}{ 1+\lk{ t[1 + \mu(B(O,d(x,O)))] }^{1-p} } \quad (0<p<1).
\end{align*}
Note that $L^{\Psi_p}(\Omega)$ is a quasi-normed space.

Let $M : = (C_\mu+1)\log \sk {e+d(x,O)} $. By \eqref{Mst} we obtain
\begin{equation}\label{Psi}
\Psi_1(x,st) \lesssim (e+d(x,O))^{-(C_\mu+1)} e^t + s \lesssim w(x) e^t + s, \quad \text{for all}\ x\in \Omega, s,t>0,
\end{equation}
where $ w: \Omega\to \real_+ $ is a weight function with
\begin{equation}\label{weight}
w(x) \lesssim \min\lk{ 1, d(x,O)^{-(C_\mu+1)} }.
\end{equation}

Let $Q^0 \in \F_0$ be the dyadic cube such that $ O \in Q^0 $. For $g \in BMO^\D(\mu)$, define
$$
\| g \|_{BMO^\D_+(\mu)} := \sup\limits_{\alpha \in \A_0} \frac{|g_{Q_{\alpha}^0}-g_{Q^0}|}{\log\sk{e+d(z_{\alpha}^0,O)}}+|g_{Q^0}| + \| g \|_{BMO^\D(\mu)},
$$
where $Q_{\alpha}^0 \in \F_0$ is a dyadic cube with its center $z_{\alpha}^0$ and $\A_0$ is the index set in Theorem \ref{dyasys}. Denote by $BMO^\D_+(\mu)$ the space consisting of all $g\in BMO^\D(\mu)$ such that $\| g \|_{BMO^\D_+(\mu)} <\8$. It is not difficult to verify that $\| \cdot \|_{BMO^\D_+(\mu)}$ is a norm on $BMO^\D_+(\mu)$.

\begin{rem} If we consider the dyadic martingales on $\real^n$, by taking appropriate cubes $Q^0$ one shows that if $g\in BMO^\D(\mu)$, then $g\in BMO^\D_+(\mu)$. Note that if $g\in BMO(\mu)$, then $g\in BMO^\D_+(\mu)$. Moreover,
$$ \| g \|_{BMO^\D_+(\mu)}\lesssim \|g\|_{BMO(\mu)}+|g_{Q^0}|. $$
\end{rem}

We now introduce the following generalized H\"{o}lder inequality for $L^1(\Omega, \F, \mu)$ and $BMO^\D_+(\mu)$.

\begin{lem}\label{Holder1}
If $f \in L^1(\Omega,\F,\mu)$ and $ g \in BMO^\D_+(\mu) $, then $f\cdot g\in L^{\Psi_1}(\Omega)$. Moreover,
\begin{equation}\label{Holder}
\| fg \|_{L^{\Psi_1}(\Omega)} \lesssim \| f \|_{1}\| g \|_{BMO^\D_+(\mu)}.
\end{equation}
\end{lem}

\begin{proof} Without loss of generality, assume $\|f\|_1\leq 1$, $\| g \|_{BMO^\D_+(\mu)}\leq 1$ and $ g_{Q^0} = 0 $. It suffices to show that
$$ \int_{\Omega} \Psi_1(x,|f(x)g(x)|) d\mu \lesssim 1. $$
	
Let $ S_k := B(O,C_0\delta^k)\setminus B(O,C_0\delta^{k+1}) $ for $k<0$ and $ S_0 := B(O,C_0) $, where $\delta \in (0,1)$ is the constant in Theorem \ref{dyasys}. Then for each $k \leq 0$, there exists a finite index subset $\B_k \subset \A_0$ such that $ B(O,C_0\delta^k) \subset \bigcup\limits_{\alpha \in \B_k}Q_{\alpha}^0 $ (where $Q^0_\alpha \in \F_0$) and 
$$ \sum\limits_{\alpha \in \B_k} \mu\sk{Q_{\alpha}^0}=\mu\sk{\bigcup\limits_{\alpha \in \B_k}Q_{\alpha}^0}  \leq  \mu\sk{B(O, 2A_0C_0\delta^{k})} \lesssim \delta^{C_\mu k} .$$
	
Take $ s= \nu^{-1}|f(x)|,t=\nu|g(x)| $ in (\ref{Psi}), one has
\begin{align*} 
\int_{\Omega} \Psi_1(x,|f(x)g(x)|) d\mu & = \sum_{k=-\8}^{0}\sum\limits_{\alpha \in \B_k} \int_{S_k\cap Q^0_{\alpha}} \Psi_1(x,|f(x)g(x)|) d\mu  
 \\
 & \lesssim \sum_{k=-\8}^{0}\sum\limits_{\alpha \in \B_k} \int_{S_k\cap Q^0_{\alpha}} w(x)e^{\nu|g(x)|}d\mu + \nu^{-1} \| f \|_1. 
\end{align*}
Therefore,
\begin{equation}\label{fg}
\int_{\Omega} \Psi_1(x,|f(x)g(x)|) d\mu \lesssim T_1 + \nu^{-1} \| f \|_1, 
\end{equation}
where
$$ T_1:= \sum_{k=-\8}^{0}\sum\limits_{\alpha \in \B_k} \int_{S_k\cap Q^0_{\alpha}} w(x)e^{\nu\jdz{g(x)-g_{Q^0_{\alpha}}}} e^{\nu\jdz{g_{Q^0_{\alpha}}}} d\mu .  $$
	
Let $ \nu := \frac{\min\{\kappa, 1\} }{2} >0 $ (where $\kappa$ is defined in Theorem \ref{MJN}), by \eqref{weight} and Theorem \ref{JN}, one has
\begin{align*} 
T_1 & \lesssim \sum_{k=-\8}^{0}\sum\limits_{\alpha \in \B_k} \dfrac{ \mu(Q^0_{\alpha}) \sk{e+d(z^0_{\alpha},O)}^{\frac{1}{2}} }{\delta^{(k+1)(C_\mu+1)}} \\
& \lesssim \sum_{k=-\8}^{0}\sum\limits_{\alpha \in \B_k} \dfrac{ \mu(Q^0_{\alpha}) \delta^{\frac{k}{2}} }{\delta^{(k+1)(C_\mu+1)}} \lesssim \sum_{k=-\8}^{0} \dfrac{\delta^{C_\mu k}\delta^{\frac{k}{2}}}{ \delta^{C_\mu k + k} } \\
& \lesssim \sum_{k=-\8}^{0} \delta^{-\frac12 k}, 
\end{align*}
and hence	
\begin{equation}\label{fg2}	
T_1 \lesssim 1.	
\end{equation}

Combine (\ref{fg}), (\ref{fg2}) and the fact that  $ \nu^{-1} \| f \|_1 \lesssim 1$, and the proof is complete.
\end{proof}

We consider the case $0<p<1$. Define
$$ \| g \|_{\Lambda_{1,+}^\D(\alpha_p)} := \sup\limits_{\alpha \in \A_0} \frac{|g_{Q_{\alpha}^0}-g_{Q^0}|}{1+\mu\lk{B\sk{O,d(z^0_\alpha,O)}}^{\alpha_p}}+|g_{Q^0}| + \| g \|_{\Lambda_1^\D(\alpha_p)}, $$
Denote by $\Lambda_{1,+}^\D(\alpha_p)$ the space consisting of all $g\in \Lambda_{1}^\D(\alpha_p)$ such that $\| g \|_{\Lambda_{1,+}^\D(\alpha_p)} <\8$. It is easy to verify that $\| \cdot \|_{\Lambda_{1,+}^\D(\alpha_p)}$ is a norm on $\Lambda_{1,+}^\D(\alpha_p)$.

\begin{rem} If we consider the dyadic martingales on $\real^n$, by taking appropriate cubes $Q^0$ one can show that if  $g\in\Lambda_{1}^\D(\alpha_p)$, then $g\in\Lambda_{1,+}^\D(\alpha_p)$.	Note that if $g\in \L_{\alpha_p}(\mu)$, then $g\in \Lambda_{1,+}^\D(\alpha_p)$. Moreover,
$$ \| g \|_{\Lambda_{1,+}^\D(\alpha_p)} \lesssim \|g\|_{\L_{\alpha_p}(\mu)}+|g_{Q^0}|.   $$
\end{rem}

Next we present a generalized H\"{o}lder inequality for $L^p(\Omega, \F, \mu)$ and $\Lambda_{1,+}^\D(\alpha_p)$ for $0<p<1$.
\begin{lemma}\label{Holder2}
If $f \in L^p(\Omega,\F, \mu)$ and $ g \in \Lambda_{1,+}^{\D}(\alpha_p) $ for $0<p<1$, then  $f\cdot g \in L^{\Psi_p}(\mu)$. Moreover,
\begin{equation}
\| fg \|_{L^{\Psi_p}(\Omega)} \lesssim \| f \|_p\| g \|_{\Lambda_{1,+}^{\D}(\alpha_p)}.
\end{equation}
\end{lemma}

\begin{proof}
Without loss of generality, assume $\|f\|_p\leq 1$, $\| g \|_{\Lambda^\D_{1,+}(\alpha_p)}\leq 1$ and $ g_{Q^0} = 0 $. It suffices to show that
$$ \int_{\Omega} \Psi_p(x,|f(x)g(x)|) d\mu \lesssim  1. $$
	
Take the same family of sets $\{ S_k \}_{k\leq 0} $ as above. From Theorem \ref{thm4.5}, we know that for $ x\in Q^0_{\alpha} $,	
\begin{align*}
|g(x)| &= | g(x) - g_{Q^0} | \leqslant \jdz{g(x)-g_{Q^0_{\alpha}} } +\jdz{ g_{Q_{\alpha}^0}-g_{Q^0}}\notag\\
& \leq  \sk{ \mu( Q^0_{\alpha} ) }^{\alpha_p} + \mu\lk{B\sk{O,d(z^0_\alpha,O)}}^{\alpha_p}+1\\
&\lesssim   \mu\sk{B(O,2A_0C_0\delta^{k})}^{\alpha_p}+1 .
\end{align*}
Therefore
\begin{align*}
\int_{\Omega} \Psi_p(x,|f(x)g(x)|) d\mu & = \sum_{k=-\8}^{0} \sum\limits_{\alpha \in \B_k} \int_{S_k\cap Q^0_{\alpha}}  \frac{ |g(x) | |f(x)|}{ 1+ \lk{|g(x) | |f(x)|\mk{1+ \mu\sk{B(O,d(x,O)} } }^{1-p}} d\mu \notag\\
& \lesssim \sum_{k=-\8}^{0} \sum\limits_{\alpha \in \B_k} \int_{S_k\cap Q^0_{\alpha}}  \frac{ |g(x) |^p |f(x)|^p}{ 1+ \mu\sk{B(O,d(x,O) }^{1-p}}  d\mu \\
 & \lesssim  \sum_{k=-\8}^{0} \sum\limits_{\alpha \in \B_k} \int_{S_k\cap Q^0_{\alpha}}  \frac{  \mu\sk{B(O,2A_0C_0\delta^{k})}^{\alpha_p p}+1  }{ \lk{ 1 + \mu\sk{B(O,C_0\delta^{k+1}} }^{1-p} }|f(x)|^p d\mu   \\
& \lesssim 1,
\end{align*}
which finishes the proof.
\end{proof}

We are now about to present the analogues of the results in Sections \ref{case1} and \ref{case2} concerning bilinear decompositions for dyadic martingales on homogeneous spaces. To this end, we need to define the product between martingale Hardy spaces and their dual spaces first. As in the probability setting, we regard the product in the sense of distribution as follows: for $0<p<1$,
$$
\la f \times g, h\ra := \la h\cdot g, f\ra,\quad f \in H^p_{\rm{at}, \D}(\mu),\  g\in \Lambda_{1,+}^\D(\alpha_p),
$$
where $ h $ is a test function such that $ h\cdot g $ is in $ \Lambda_{1,+}^\D(\alpha_p) $. For $p=1$, we may define the product between $H^1_{\rm{at}, \D}(\mu)$ and $BMO^\D(\mu)$ analogously. To this end, we need to introduce some pointwise multipliers of $ \Lambda_{1,+}^\D(\alpha_p)$ and $BMO^\D(\mu)$.

Denote the space of test functions by $ \H(\alpha_p) $ $(0<p\leq 1)$, and a measurable function $ h $ is a test function if it satisfies the following properties:
\begin{equation}
| h(x) | \lesssim \frac{1}{ \sk{ 1+\mu (B(O,d(x, O)))^{\alpha_p} } \log(e+ d(x, O) ) }, \quad \forall x\in \Omega,
\end{equation}
and
\begin{equation}
|h(y)-h(z)| \lesssim \frac{\mu( B )^{\alpha_p}}{ \sk{1+ \mu[B(O,1+r+d(c_B, O))]^{\alpha_p} } \log(e+r+d(c_B, O))  }
\end{equation}
whenever $ y,z $ are both contained in a ball $ B $  with center $ c_B $ and radius $ r \leq \frac{d(c_B, O)}{2A_0}+1 $.

It is obvious that $ \H(\alpha_p) \subset L^{\infty}(\Omega) $. The following theorem shows that if  $h\in \H(\alpha_p)$, then $h$ is a pointwise multiplier of $ \Lambda_{1,+}^\D(\alpha_p)$.

\begin{thm}\label{thm7.4}
For $ 0<p<1 $ and any dyadic system $\D$, $ \H(\alpha_p) $ is a space of pointwise multipliers of $ \Lambda_{1,+}^\D(\alpha_p) $. For $ p=1 $, $ \H(0) $ is a space of pointwise multipliers of $ BMO^\D_+(\mu) $. More precisely, for any $ g \in \Lambda_{1,+}^{\D}(\alpha_p) $ and $ h \in \H(\alpha_p) $, we have
$$ \| g\cdot h \|_{\Lambda_{1,+}^{\D}(\alpha_p)} \lesssim \| g \|_{\Lambda_{1,+}^{\D}(\alpha_p)} \sk{ \| h \|_{L^{\infty}(\Omega)}+1},$$
and for any $ g \in BMO^\D_+(\mu) $ and $ h \in \H(0) $, we have
$$ \| g\cdot h \|_{BMO^\D_+(\mu)} \lesssim \| g \|_{BMO^\D_+(\mu)} \sk {\| h \|_{L^{\infty}(\Omega)}+1}. $$
\end{thm}

\begin{proof} First, we consider the case $0<p<1$. Assume that $ g \in \Lambda_{1,+}^{\D}(\alpha_p) $ and $ h \in \H(\alpha_p) $. According to \cite{Na}, it suffices to show that
\begin{equation}\label{PWM}
\sup_{Q}\frac{|g_Q|}{\mu(Q)^{\alpha_p+1}} \sk{ \int_Q |h(x)-h_Q| dx } < \infty,
\end{equation}
where $Q$ runs over all dyadic cubes in $ \D $.

If $ Q \subset Q^0_\beta $ for some $\beta\in \A_0$, there exists a collection of cubes $ Q = Q_0 \subset Q_1 \subset \cdots \subset Q_N = Q^0_\beta $ such that there exists a universal constant $0<\delta^{'}<1$ with $ \mu(Q_{k-1}) \leq \delta^{'} \mu (Q_k)$. Hence
 \begin{align*}
 |g_Q-g_{Q^0_\beta}| & \leq \sum_{k=1}^{N}| g_{Q_k} - g_{Q_{k-1}} | \lesssim \sum_{k=1}^{N} \mu(Q_k)^{\alpha_p} \| g \|_{\Lambda_{1,+}^{\D}(\alpha_p)} \\
& \lesssim \| g \|_{\Lambda_{1,+}^{\D}(\alpha_p)} \sum_{k=1}^{N} \int^{\mu(Q_k) }_{\mu(Q_{k-1})}  t^{\alpha_p-1} dt   \\
& \lesssim \mu(Q^0_\beta)^{\alpha_p} \| g \|_{\Lambda_{1,+}^{\D}(\alpha_p)}.
 \end{align*}
Similarly, if $ Q_\beta^0 \subset Q $, we have
$$ | g_Q - g_{Q^0_\beta} | \lesssim \mu(Q)^{\alpha_p} \| g \|_{\Lambda_{1,+}^{\D}(\alpha_p)}. $$

By Theorem \ref{dyasys}, there exists a ball $ B $, with center $c_B$ and radius $r$, such that $ Q \subset B $ and $ \mu(B) \lesssim \mu(Q) $. 
  
 If $ Q^0_\beta \subset Q $ and $ r > \frac{d(O, c_B)}{2A_0} + 1 $, for any $ x \in B(O, r) $, we have $ d(c_B, x) \leq A_0( d(c_B, O) + d(O, x) ) < (2A_0^2+A_0) r $. Then $ \mu(Q) \gtrsim \mu(B) \gtrsim C_{\mu}^{ -(2A_0^2+A_0) } \mu \sk{ B(O, r) } \gtrsim 1 $. Similarly, we also have $ d(z_{\beta}^0, O) < (2A_0^2+A_0) r $ and $ \mu\lk{B\sk{O,d(z^0_\beta,O)}} \lesssim \mu(B) \lesssim \mu(Q) $. Thus
\begin{align*}
 \frac{|g_Q|}{\mu(Q)^{\alpha_p+1}} \sk{ \int_Q |h(x)-h_Q| dx } & \lesssim \frac{|g_Q-g_{Q_{\beta}^0}|+ |g_{Q_{\beta}^0}-g_{Q^0}| + |g_{Q^0}| }{\mu(Q)^{\alpha_p}} \cdot \| h \|_{ L^{\infty}(\Omega) } \\
  & \lesssim \frac{ \mu(Q)^{\alpha_p} +\mu\lk{B\sk{O,d(z^0_\beta,O)}}^{\alpha_p}+1  }{ \mu(Q)^{\alpha_p} } \cdot \| g \|_{\Lambda_{1,+}^{\D}(\alpha_p)}  \| h \|_{L^{\infty}(\Omega)} \\
 & \lesssim \| g \|_{\Lambda_{1,+}^{\D}(\alpha_p)}  \| h \|_{L^{\infty}(\Omega)}. 
\end{align*}

If $ Q^0_\beta \subset Q $ and $ r \leq \frac{d(O, c_B)}{2A_0} +1 $, for any $ x \in B $, we have $ d(x,O) \leq A_0 (d(O,c_B) + r) $, then $ \mu(Q) \lesssim \mu \sk{ B(O, A_0( d(O, c_B)+r ) } $. Thus
\begin{align*}
\frac{|g_Q|}{\mu(Q)^{\alpha_p+1}} \sk{ \int_Q |h(x)-h_Q| dx }  & \lesssim \frac{|g_Q-g_{Q_{\beta}^0}|+ |g_{Q_{\beta}^0}-g_{Q^0}| + |g_{Q^0}| }{\mu(Q)^{\alpha_p+1}} \frac{\mu( B )^{\alpha_p+1}}{ \sk{ 1+\mu[B(O,1+r+d(O,c_B))]  }^{\alpha_p} } \\
& \lesssim \frac{ \sk{ \mu(Q)^{\alpha_p} +\mu\lk{B\sk{O,d(z^0_\beta,O)}}^{\alpha_p}+1 } \| g \|_{\Lambda_{1,+}^{\D}(\alpha_p)}  }{ \sk{ 1+\mu[B(O,1+r+d(O,c_B))]  }^{\alpha_p} } \\
& \lesssim \| g \|_{\Lambda_{1,+}^{\D}(\alpha_p)}.
\end{align*}

If $Q \subset Q^0_\beta $, from Theorem \ref{dyasys}, we can choose $ C_0 $ sufficiently small such that $ C_1 = 2A_0C_0 \leq 1$, then $ r \leq C_1 \leq \frac{d(c_B,O)}{2A_0}+1 $. For any $ x \in Q^0_\beta $, we have $ d(O ,x) \leq A_0( d(O, z^0_\beta) + C_1 ) $. Then
$$ \mu(Q^0_\beta) \lesssim \mu \lk{ B\sk{ O, A_0( d(O, z_\beta^0) + C_1 ) } }. $$ 
By a calculation similar to the one presented above, we get the desired result. 

Combining the above estimates, we finish our proof for $ 0<p<1 $. The case for $ p=1 $ is similar.
\end{proof}

\begin{remark}
Note that in Theorem \ref{thm7.4}, the dyadic system $\D$ is arbitrary. Then from Theorem \ref{Lips} and (\ref{bmo}),  we conclude that $ \H(\alpha_p) $ is a space of pointwise multipliers of $ \L_{\alpha_p}(\mu) $ and $ \H(0) $ is a space of pointwise multipliers of $ BMO(\mu) $.
\end{remark}

\subsection{Bilinear decompositions}\label{bd_hs} Assume $ f \in H^1_{\D}(\mu), g \in BMO^{\D}_+(\mu) $ or $f\in H^p_{ \D}(\mu), g\in \Lambda^\D_{1,+}(\alpha_p) $, $0<p<1$. 

Denote by $ \mathcal{H}^p_{\D, \text{fin}} (\mu) \ (0<p\leq 1)$ the linear space consisting of all functions which can be written as a finite sum of simple $(p, \8)$-atoms. Thus if $f\in \mathcal{H}^p_{\D, \text{fin}} (\mu)$, $f$ is locally supported, $f\in L^1(\Omega)\cap L^\8(\Omega)$ and $\int_{\Omega}fd\mu=0$. Note that $ \mathcal{H}^p_{\D, \text{fin}} (\mu)$ is dense in $H^p_{\D}(\mu)$ with respect to the norm $\|\cdot \|_{H^p_\D(\mu)}$. 

In the following, we shall only consider the case where $f\in  \mathcal{H}^p_{\D, \text{fin}} (\mu)$. Then $f\cdot g\in L^1(\Omega)$, and we can write
\begin{equation}\label{dec1}
 f\cdot g  = \Pi_1(f,g) + \Pi_2(f,g) + \Pi_3(f,g), 
\end{equation}
where
$$ \Pi_1(f,g) : =  \sum_{k=1}^{\infty} d_k f d_k g, \quad \Pi_2(f,g) : = \sum_{k=1}^{\infty} f_{k-1}d_k g  \quad \text{and} \quad \Pi_3(f,g) : =  \sum_{k=1}^{\infty} g_{k-1}d_k f. $$

\begin{thm}\label{thm7.5} We have the following:
\begin{enumerate}
\item $\Pi_1$ is a bilinear bounded operator from $H^1_{\D}(\mu)\times BMO^{\D}_+(\mu)$ to $L^1(\Omega)$,  and \\
when $0<p<1$, $\Pi_1$ is a bilinear bounded operator from $H^p_{\D}(\mu)\times \Lambda_{1,+}^{\D}(\alpha_p)$ to $L^1(\Omega)$.
\item $\Pi_2$ is a bilinear bounded operator from $H^1_{\D}(\mu)\times BMO^{\D}_+(\mu)$ to $H^1_\D(\mu)$,  and \\
when $0<p<1$, $\Pi_2$ is a bilinear bounded operator from $H^p_{\D}(\mu)\times \Lambda_{1,+}^{\D}(\alpha_p)$ to $H^1_\D(\mu)$.
\item $\Pi_3$ is a bilinear bounded operator from $H^1_{\D}(\mu)\times BMO^{\D}_+(\mu)$ to $H^{\Psi_1}_\D(\mu)$,  and \\ when $0<p<1$, $\Pi_3$ is a bilinear bounded operator from $H^p_{\D}(\mu)\times \Lambda_{1,+}^{\D}(\alpha_p)$ to $H^{\Psi_p}_\D(\mu)$.
\end{enumerate}
\end{thm}

\begin{proof} For $\Pi_1$ and $\Pi_2$, we can argue as in the corresponding part of the proof of Theorem \ref{Theorem A}. As for $\Pi_3$, we can also argue as in the corresponding part of the proof Theorem \ref{Theorem A}, where in the homogeneous setting one needs to apply Lemma \ref{Holder1} and Lemma \ref{Holder2}. We omit the details.
\end{proof}

\begin{rem} For $\Pi_1$ and $\Pi_2$, the condition $H^1_{\D}(\mu)\times BMO^{\D}_+(\mu)$ and $H^p_{\D}(\mu)\times \Lambda_{1,+}^{\D}(\alpha_p)$ can be in fact replaced by $H^1_{\D}(\mu)\times BMO^{\D}(\mu)$ and $H^p_{\D}(\mu)\times \Lambda_{1}^{\D}(\alpha_p)$, respectively.
\end{rem}

\section{Applications to Homogeneous spaces}\label{section8}
In the first part of this section we show that $ H^{\Psi_p}_\D(\mu) $ admits an atomic decomposition for $0<p<1$, which allows us to integrate several adjacent dyadic systems on homogeneous spaces. 

For a given dyadic system $ \D$ on $\Omega$, we define the dyadic $ H^{\Psi_p}_{\rm at,\D} $-atom as follows.

\begin{definition} A measurable function $a$ is said to be an $H^{\Psi_p}_{\rm at,\D}$-atom if
\begin{enumerate}[(i)]
\item $ \supp(a) \subset Q $ where $ Q \in \D $ is a cube;
\item $\int_{\Omega}a d\mu = 0$;
\item $ \| a\|_{\infty} \leqslant \| 1_Q \|_{L^{\Psi_p}(\Omega)}^{-1} $.
\end{enumerate}
\end{definition}
The atomic dyadic martingale Musielak--Orlicz Hardy spaces $H^{\Psi_p}_{\rm at,\D}(\mu) \ (0<p< 1)$ are defined in a way analogous to \eqref{De1} and \eqref{De2}. We first introduce the space $ BMO_{\Psi_p}^\D(\mu) $, which is a subspace of continuous linear functionals on finite sums of atoms.

\begin{definition} A locally integrable function $ g $ is said to be a dyadic $ BMO_{\Psi_p}^\D(\mu) $ function associated with a dyadic system $ \D $ if
$$ \| g \|_{BMO_{\Psi_p}^\D(\mu)} := \sup_{k \in \ent}\sup_{Q \in \F_k} \frac{1}{\| 1_{Q} \|_{L^{\Psi_p}(\Omega)}} \int_Q |g(x)-g_k(x)| dx  < \infty. $$
\end{definition}

Then we define the atomic Musielak--Orlicz martingale Hardy spaces $ H^{\Psi_p}_{\rm at,\D}(\mu) $ as follows:
\begin{multline*}
H^{\Psi_p}_{\rm at,\D}(\mu) := \\
\left\{ f \in \sk{BMO_{\Psi_p}^\D(\mu)}^* : \ f = \sum_{i=0}^{\infty} \lambda_i a_{i}, \text{ where $ a_{i} $ is an $ H^{\Psi_p}_{\rm at,\D}(\mu) $-atom supported on a cube $ Q_i. $} \right\},
\end{multline*}
where
$$ \sum_{i=0}^{\infty} \int_{Q_i}\Psi_p(x,|\lambda_i|\| a_{i} \|_{\infty})d\mu < \infty. $$
Moreover,
$$ \| f \|_{H^{\Psi_p}_{\rm at,\D}(\mu)} := \inf\lk{ \rho>0 : \sum_{i=0}^{\infty} \int_{Q_i}\Psi_p(x,\rho^{-1}|\lambda_i|\| a_{i} \|_{\infty})d\mu \leqslant 1} $$

Arguing as in \cite{XJY}, one can show that for $0<p< 1$
\begin{equation}\label{atp}
H^{\Psi_p}_{\D}(\mu) = H^{\Psi_p}_{\rm at,\D}(\mu).
\end{equation}

We shall now introduce the atomic Musielak--Orlicz Hardy spaces $ H^{\Psi_p}_{\rm at}(\mu) \ (0<p\leq 1)$ on the homogeneous space $\Omega$. First, we present the definition of atoms for $ H^{\Psi_p}_{\rm at}(\mu) $.

\begin{definition}
A measurable function $a(x)$ is said to be an $H^{\Psi_p}_{\rm at}(\mu)$-atom if
\begin{enumerate}[(i)]
\item $ \supp(a) \subset B $ where $ B \subset \Omega $ is a ball;
\item $\int_{\Omega}a d\mu = 0$;
\item $ \| a \|_{\infty} \leqslant \| 1_B \|_{L^{\Psi_p}(\Omega)}^{-1} $.
\end{enumerate}
\end{definition}

\begin{definition}
A locally integrable function $ g $ is said to be a $ BMO_{\Psi_p}(\mu) $ function  if
$$
\| g \|_{BMO_{\Psi_p}(\mu)} := \sup_{B} \frac{1}{\| 1_{B} \|_{L^{\Psi_p}(\Omega)}} \int_B |g(x)-g_B|dx < \infty,
$$
where $B$ runs over all balls in $ \Omega $.
\end{definition}
	
\begin{definition}
The atomic Musielak--Orlicz Hardy spaces $H^{\Psi_p}_{\rm at} (\mu)\ (0<p\leq 1)$ are defined as follows:
\begin{multline*}
H^{\Psi_p}_{\rm at} (\mu) := \\
\left\{ f \in \sk{BMO_{\Psi_p}(\mu)}^* : f = \sum_{i=0}^{\infty} \lambda_i a_{i}, \text{ where $ a_{i} $ is an $ H^{\Psi_p}_{\rm at}(\mu) $-atom supported on a ball $B_i$} \right\}, 
\end{multline*}
where
$$ \sum_{i=0}^{\infty} \int_{B_i}\Psi_p(x,|\lambda_i|\| a_{i} \|_{\infty})d\mu < \infty. $$
Moreover,
$$ \| f \|_{H^{\Psi_p}_{\rm at}(\mu)} := \inf\lk{ \rho>0 : \sum_{i=0}^{\infty} \int_{B_i}\Psi_p(x,\rho^{-1}|\lambda_i|\| a_{i} \|_{\infty})d\mu \leqslant 1} . $$	
\end{definition}

Let $\D^t \ (1\leq t\leq K)$ be the adjacent systems of Theorem \ref{kdya}. By arguing as in the proof of Theorem \ref{Hardy}, we have the following:
\begin{lemma}\label{5.10} For $0<p< 1$,
$ H^{\Psi_p}_{\rm at}(\mu) = H^{\Psi_p}_{\rm at, \D^1} (\mu)+ H^{\Psi_p}_{\rm at, \D^2}(\mu) + \cdots + H^{\Psi_p}_{\rm at, \D^K}(\mu) $.
\end{lemma}

\begin{proof}

It suffices to show that any dyadic $ H^{\Psi_p}_{\D^t} $-atom  $ a $ is a constant multiple of an $ H^{\Psi_p}(\mu) $-atom, and any $ H^{\Psi_p}(\mu) $-atom $ b $ is a constant multiple of a dyadic $ H^{\Psi_p}_{\D^t} $-atom.
	
If $B:= B(x_0,r)$, then denote the ball $ B(x_0,Dr) $ by $DB$ for $D\geq 1$. Denote $ d: =d(x_0, O) $. In what follows, $C(D,p,A_0,C_\mu)$ denotes a constant that depends on $D,p,A_0,C_\mu$ and may differ from line to line.
We first show that if  
$$ \int_{B}\frac{1}{ 1+[1 + \mu(B(O,d(x,O)))]^{1-p} } d\mu(x) = 1, $$ then
\begin{equation}\label{Ball}
\int_{DB}\frac{1}{ 1+[1 + \mu(B(O,d(x,O)))]^{1-p} } d\mu(x) \leqslant C(D,p,A_0,C_\mu) .
\end{equation}
	
Notice that
\begin{align*}
1 &= \int_{B}\frac{1}{ 1+[1 + \mu(B(O,d(x,O)))]^{1-p} } d\mu(x) \notag \\
& \geqslant \frac{\mu(B)}{\sup\limits_{x\in B} \lk{ 1+[1 + \mu(B(O,d(x,O)))]^{1-p}} } \notag\\
& \geqslant \frac{\mu(B)}{  1+[1 + \mu(B(O,A_0(d+r)))]^{1-p}  },
\end{align*}
which implies
\begin{equation*}
\mu(B) \leqslant 1+[1 + \mu(B(O,A_0(d+r)))]^{1-p}.
\end{equation*}
	
If $ d \leqslant 2A_0Dr $, we have
\begin{align*}
\mu(B) & \leqslant 1 + [1 + \mu(B(O,A_0(2A_0D+1)r))]^{1-p} \\
&\leqslant 1+ \lk{1 + \mu[B(x_0,A_0(A_0+1)(2A_0D+1)r)]}^{1-p} \\
 & \leqslant 1 + \lk{1 + \mk{A_0( A_0 +1 )(2A_0D+1)}^{C_\mu}\mu(B)}^{1-p},
\end{align*}
and thus $ \mu(B) \leqslant C(D,p,A_0,C_\mu). $
	
Then
\begin{equation}\label{dsmall}
\int_{DB}\frac{1}{ 1+[1 + \mu(B(O,d(x,O)))]^{1-p} } d\mu(x) \leqslant  \mu(DB)
\leqslant  D^{C_\mu}\mu(B) \leqslant C(D,p,A_0,C_\mu).
\end{equation}
	
If $d > 2A_0Dr$, then
\begin{align*}
&\int_{DB}\frac{1}{ 1+[1 + \mu(B(O,d(x,O)))]^{1-p} } d\mu(x) \notag\\
\leqslant & \frac{\mu(DB)}{\inf\limits_{x\in DB} \lk{1+[1 + \mu(B(O,d(x,O)))]^{1-p}}} \notag\\
\leqslant & \frac{D^{C_\mu}\mu(B)}{ 1+[1 + \mu(B(O,d/A_0-Dr))]^{1-p} } , \notag\\
\leqslant &  \frac{D^{C_\mu}\lk{ 1+\mu[B\sk{O,[A_0+1/(2D)]d}] }^{1-p} + D^{C_\mu} }{ 1+\lk{1 + \mu[B(O,d/(2A_0))] }^{1-p}  } \notag\\
\leqslant & \frac{D^{C_\mu}\lk{ 1+ \mk{ (2A_0+1/D)A_0 }^{C_\mu}\mu[B(O,d/(2A_0))] }^{1-p} }{ 1+\lk{1 + \mu[B(O,d/(2A_0))] }^{1-p} } + D^{C_\mu}.
\end{align*}
Hence
\begin{equation}\label{dbig}
\int_{DB}\frac{1}{ 1+[1 + \mu(B(O,d(x,O)))]^{1-p} } d\mu(x) \leqslant  C(D,p,A_0,C_\mu)
\end{equation}
	
Combining \eqref{dsmall} with \eqref{dbig}, we get (\ref{Ball}).
	
Assume $ a $ is an $ H^{\Psi_p}(\mu) $-atom supported on $ B $. By Theorem \ref{kdya}, there exist $ t $ and a cube $ Q \in \D^t $ such that $ B \subset Q $ and $  \mathrm{diam}  (Q) \leqslant Cr $, hence $ B\subset Q \subset CB $.
	
Note that $  \supp(a )  \subset Q, \int_Q a (x) d\mu (x)= 0 $ and
$$ \| 1_Q \|_{L^{\Psi_p}(\mu)} \leqslant \| 1_{CB} \|_{L^{\Psi_p}(\mu)} \leqslant C(C,p, A_0,C_\mu)\| 1_{B} \|_{L^{\Psi_p}(\mu)}, $$
which follows from \eqref{Ball}. Thus
$$ \| a \|_{\infty} \leqslant \| 1_{B} \|_{L^{\Psi_p}(\mu)}^{-1} \lesssim \| 1_Q \|_{L^{\Psi_p}(\mu)}^{-1}, $$
which implies $ a $ is a multiple of dyadic $  H^{\Psi_p}_{\D^t} $-atom supported on $Q$.
	
For any $ t =1,2 \cdots,K $, assume $ b $ is a dyadic $ H^{\Psi_p}_{\D^t} $-atom supported on ${Q^k_{\beta}} $. By Theorem \ref{dyasys}, there exists two balls such that $ B(z_{\beta}^k,c_1\delta^k) \subset Q^k_{\beta} \subset B(z_{\beta}^k,C_1\delta^k)$.
	
Thus $ \supp ( b ) \subset B(z_{\beta}^k,C_1\delta^k), \int_{B(z_{\beta}^k,C_1\delta^k)}  b (x) d\mu (x) = 0 $ and
$$ \| 1_{B(z_{\beta}^k,C_1\delta^k)} \|_{L^{\Psi_p}(\mu)} \leqslant C\sk{\frac{C_1}{c_1},p,A_0,C_\mu}\| 1_{B(z_{\beta}^k,c_1\delta^k)} \|_{L^{\Psi_p}(\mu)} \lesssim \| 1_{Q^k_{\beta}} \|_{L^{\Psi_p}(\mu)}, $$
which follows from \eqref{Ball}. Therefore,
$$ \|  b \|_{\infty} \leqslant \| 1_{Q^k_{\beta}} \|_{L^{\Psi_p}(\mu)}^{-1} \lesssim  \| 1_{B(z_{\beta}^k,C_1\delta^k)} \|_{L^{\Psi_p}(\mu)}^{-1} , $$
which implies $  b $ is a multiple of dyadic $H^{\Psi_p}$-atom supported on $B(z_{\beta}^k,C_1\delta^k)$.
\end{proof}

\begin{remark}
In \cite{FMY}, Fu, Ma and Yang defined another kind of Musielak--Orlicz Hardy spaces by grand maximal function and they also proved that these Musielak--Orlicz Hardy spaces are equivalent to $ H^{\Psi_p}_{\rm at}(\mu) $ with respect to the corresponding norms when $p\in (\frac{C_\mu}{C_\mu+1}, 1]$.
\end{remark}

Let $B_1:=B(O, 1)$. Define
$$  \|g\|_{BMO_+(\mu)}:= |g_{B_1}|  + \|g\|_{BMO(\mu)}, \quad \text{for } g\in BMO(\mu), $$
and
$$ \|g\|_{\L_{+, \alpha}(\mu)}:= |g_{B_1}|+ \|g\|_{\L_{ \alpha_p}(\mu)},\quad \text{for } g\in \L_{ \alpha_p}(\mu). $$
Thus $\|\cdot\|_{BMO_+(\mu)}$ and $\|\cdot \|_{\L_{+, \alpha_p}(\mu)}$ are quasi-norms on $BMO(\mu)$ and $\L_{ \alpha_p}(\mu)$, respectively.

\begin{thm} 
Let $0<p<1$ and $f\in H^p_{\rm at}(\mu)$. There exist three linear continuous operators $\Pi^f_1: \L_{\alpha_p}(\mu)\rightarrow L^1(\Omega)$, $\Pi^f_2: \L_{\alpha_p}(\mu)\rightarrow H^1_{\rm at}(\mu)$ and $\Pi^f_3: \L_{\alpha_p}(\mu)\rightarrow H^{\Psi_p}_{\rm at}(\mu)$ such that
$$  f\cdot g= \Pi^f_1(g)+\Pi^f_2(g)+\Pi^f_3(g) \quad \text{for all } g\in \L_{\alpha_p}(\mu), $$
where $\L_{ \alpha_p}(\mu)$ is endowed with the quasi-norm $\|\cdot \|_{\L_{+, \alpha_p}(\mu)}$.
\end{thm}

\begin{proof} Let $ f \in H^p_{\rm at}(\mu) $. By Theorem  \ref{Hardy} there exist $ f^t \in H^p_{\D^t}(\mu)\ (t = 1,2,\cdots,K) $ such that $ f = f^1 + f^2 + \cdots + f^K $, and
$$  \sum_{t=1}^{K} \|f^t\|_{H^p_{\D^t}(\mu)}\approx \|f\|_{H^p_{\rm at}(\mu)}. $$
Define $\Pi_i^f(g) :=\sum\limits_{t=1}^{K} \Pi_i(f^t, g)$ for $i=1,2,3$ and $g\in \L_{ \alpha_p}(\mu)$ ($\Pi_i$ defined as in Theorem \ref{thm7.5}). Then
\begin{equation*}
f\cdot g = \Pi_1^f(g) + \Pi_2^f(g) + \Pi_3^f(g).
\end{equation*}

By Theorem \ref{thm7.5}, Theorem \ref{Hardy} and Lemma \ref{5.10}, we have
\begin{align*}
&\| \Pi_1^f(g) \|_1 \lesssim \sum_{t=1}^{K}\| \Pi_1(f^t, g)\|_{1}\lesssim  \sum_{t=1}^K
\| f^t \|_{H^p_{\D^t}(\mu)} \| g \|_{\Lambda_{1,+}^{\D^t}(\alpha_p)} \lesssim  \| f \|_{H^p_{\rm at}(\mu)} \| g \|_{\L_{+, \alpha_p}(\mu)} ,\\
&\| \Pi_2^f(g) \|_{H^1_{\rm at}(\mu)} \lesssim \sum_{t=1}^{K}\| \Pi_2(f^t, g)\|_{H^1_{\D^t}(\mu)}\lesssim  \sum_{t=1}^K
\| f^t \|_{H^p_{\D^t}(\mu)} \| g \|_{\Lambda_{1,+}^{\D^t}(\alpha_p)} \lesssim  \| f \|_{H^p_{\rm at}(\mu)} \| g \|_{\L_{+, \alpha_p}(\mu)},\\
&\| \Pi_3^f(g) \|_{H^{\Psi_p}_{\rm at}(\mu)} \lesssim \sum_{t=1}^K\| \Pi_3(f^t, g)\|_{H^{\Psi_p}_{\D^t}(\mu)}\lesssim \sum_{t=1}^K
\| f^t \|_{H^p_{\D^t}(\mu)} \| g \|_{\Lambda_{1,+}^{\D^t}(\alpha_p)} \lesssim  \| f \|_{H^p_{\rm at}(\mu)} \| g \|_{\L_{+, \alpha_p}(\mu)}.
\end{align*}
which finishes the proof.
\end{proof}

\begin{rem}
If the homogeneous space $(\Omega, \mu)$ satisfies the reverse doubling condition, then Lemma \ref{5.10} holds for $p=1$. Then we conclude the following.

Let $f\in H^1_{\rm at}(\mu)$. There exist three linear continuous operators $\Pi^f_1: BMO(\mu)\rightarrow L^1(\Omega)$, $\Pi^f_2: BMO(\mu)\rightarrow H^1_{\rm at}(\mu)$ and $\Pi^f_3: BMO(\mu)\rightarrow H^{\Psi_1}_{\rm at}(\mu)$ such that
$$  f\cdot g= \Pi^f_1(g)+\Pi^f_2(g)+\Pi^f_3(g)\quad \text{for all } g\in BMO(\mu),  $$
where $BMO(\mu)$ is endowed with the norm $\|\cdot\|_{BMO_+(\mu)}$.
\end{rem}

\bigskip

\subsection*{Acknowledgments} We thank Professor Quanhua Xu for helpful discussions and suggestions. O.B. and Y.Z. would like to express their gratitude to Professor Xu for his kind invitation and hospitality during their visit to Besan\c{c}on in March 2022. 

We would also like to thank Yong Jiao, Guangheng Xie, Dachun Yang, Dejian Zhou for personal communications on their work with us.

\end{document}